\documentclass[reqno]{amsart}
\usepackage{amssymb, amsmath, amsthm, enumerate, mathtools, graphicx, enumitem, tikz, color, soul, hyperref}
\usepackage{stmaryrd}
\usepackage[margin = 1.5 in]{geometry}
\usepackage[mathscr]{euscript}
\allowdisplaybreaks

\newcommand{\fE}{\mathcal{E}}

\newcommand{\fS}{\mathcal{S}}
\newcommand{\fT}{\mathcal{T}}
\newcommand{\fR}{\mathcal{R}}
\newcommand{\fI}{\mathcal{I}}
\newcommand{\fJ}{\mathcal{J}}
\newcommand{\C}{\mathbb{C}}
\newcommand{\R}{\mathbb{R}}

\newcommand{\Z}{\mathbb{Z}}

\newcommand{\bbT}{\mathbb{T}}

\newcommand{\supp}{\operatorname{supp}}

\newcommand{\fB}{\mathcal{B}}

\newcommand{\fQ}{\mathcal{Q}}

\newcommand{\dist}{\operatorname{dist}}

\newcommand{\Br}{\operatorname{Br}}
\newcommand{\Bil}{\operatorname{Bil}}
\newcommand{\ldist}{\operatorname{dist_L}}
\newcommand{\diam}{\operatorname{diam}}
\newcommand{\negl}{\operatorname{negligible}}

\def\Xint#1{\mathchoice
   {\XXint\displaystyle\textstyle{#1}}%
   {\XXint\textstyle\scriptstyle{#1}}%
   {\XXint\scriptstyle\scriptscriptstyle{#1}}%
   {\XXint\scriptscriptstyle\scriptscriptstyle{#1}}%
   \!\int}
\def\XXint#1#2#3{{\setbox0=\hbox{$#1{#2#3}{\int}$}
     \vcenter{\hbox{$#2#3$}}\kern-.5\wd0}}

\def\dashint{\Xint-}

\newtheorem{theorem}{Theorem}[section]
\newtheorem{proposition}[theorem]{Proposition}
\newtheorem{lemma}[theorem]{Lemma}

\theoremstyle{definition}
\newtheorem{definition}[theorem]{Definition}

\newtheorem{remark}[theorem]{Remark}

\numberwithin{equation}{section}

\begin{document}
\title[Local extension estimates for the hyperbolic hyperboloid]{Local extension estimates for the hyperbolic hyperboloid in three dimensions}
\address{Department of Mathematics, University of Wisconsin, Madison}
\email{bbruce@math.wisc.edu}
\author{Benjamin Baker Bruce}
\date{\today}
\maketitle

\begin{abstract}
We establish Fourier extension estimates for compact subsets of the hyperbolic hyperboloid in three dimensions via polynomial partitioning.
\end{abstract}

\section{Introduction}
In this article, we establish Fourier extension estimates for compact subsets of the hyperbolic, or one-sheeted, hyperboloid in three dimensions.  This surface may be defined as the set of points $(\tau,\xi) \in \R \times \R^2$ satisfying the relation $\tau^2 = 1 + \xi_1^2-\xi_2^2$.  Setting $\phi(\xi) := \sqrt{1+\xi_1^2-\xi_2^2}$ and $\Omega := \{\xi \in \R^2 : 1+\xi_1^2-\xi_2^2 \geq 0\}$, we will restrict our attention to the graph
\begin{align*}
\Sigma := \{(\phi(\xi),\xi) : \xi \in \Omega\}.
\end{align*}   
We aim to adapt the polynomial partitioning method of Guth \cite{Guth} to obtain extension estimates for a bounded subset of $\Sigma$ near $(1,0)$, which we denote by $\Sigma_1$.  Use of the parabolic scalings $P_r(\tau,\xi) := (r^{-2}\tau,r^{-1}\xi)$ in Guth's argument presents an apparent obstacle here, as hyperboloids are evidently not preserved by such transformations.  To overcome this minor issue, we will simultaneously prove extension estimates for all parabolic rescalings of $\Sigma_1$ with constants uniform in the scaling parameter.  Toward that end, let $U := \{\xi : |\xi| \leq \delta_0/10\}$, where $\delta_0 > 0$ is a small constant to be chosen later, and for each $r \in (0,1]$, let $\phi_r(\xi) := r^{-2}(\phi(r\xi)-1)$ and
\begin{align*}
\Sigma_r := \{(\phi_r(\xi),\xi) : \xi \in U\}.
\end{align*}
Each $\Sigma_r$ is the image of $\Sigma_1 \cap \{(\tau,\xi) : \xi \in rU\}$ under the parabolic scaling $P_r$, and the ``$-1$" in $\phi_r$ just makes $\Sigma_r$ converge to the hyperbolic paraboloid $\Sigma_0 := \{(\frac{1}{2}(\xi_1^2-\xi_2^2),\xi) : \xi \in U\}$ as $r\rightarrow 0$.   We associate to  $\Sigma_r$ the extension operator
\begin{align*}
\fE_r f(t,x) := \int_{U} e^{2\pi i(t,x)\cdot(\phi_r(\xi),\xi)}f(\xi)d\xi.
\end{align*}

\begin{theorem}\label{thm1.1}
If $q > 13/4$ and $p > (q/2)'$, then $\fE_r : L^p(U) \rightarrow L^q(\R^3)$ with operator norm bounded uniformly in $r$.
\end{theorem}

\begin{remark}
The bilinear and bilinear-to-linear theories for $\fE_1$ appear in a separate preprint \cite{Bruce--Oliveira e Silva--Stovall} of Stovall, Oliveira e Silva, and the author.  Using the bilinear machinery and Theorem \ref{thm1.1}, boundedness of $\fE_r$ on the parabolic scaling line $p = (q/2)'$ (for $q > 13/4$) can also be proved.  See \cite[Remark 5.2]{Bruce--Oliveira e Silva--Stovall}, as well as \cite{Stovall}, \cite{Tao--Vargas--Vega}, and \cite{Kim} for arguments of this type.   
\end{remark}

Theorem \ref{thm1.1} can be compared to several recent developments in the restriction/extension theory for hyperbolic surfaces in three dimensions.  Cho and Lee \cite{Cho--Lee} generalized Guth's argument in \cite{Guth} to the hyperbolic paraboloid, proving strong type $(p,q)$ extension estimates in the range $q > 13/4$, $p \geq q$.  Later work of Kim \cite{Kim} and Stovall \cite{Stovall} brought those estimates to the scaling line $p =(q/2)'$.  (Letting $r \rightarrow 0$ and applying Fatou's lemma, Theorem \ref{thm1.1} reproves the off-scaling extension estimates for the hyperbolic paraboloid.)  Recently Buschenhenke, M\"uller, and Vargas \cite{Buschenhenke--Muller--Vargas} used polynomial partitioning to obtain extension estimates for a one-variable perturbation of the hyperbolic paraboloid;  Guo and Oh \cite{Guo--Oh} have now obtained similar estimates for general polynomial perturbations. Near the origin, the hyperbolic hyperboloid can be viewed as a (rather special) non-polynomial perturbation of the hyperbolic paraboloid.

The rest of the article is organized as follows:  In Section 2, we adapt the notion of ``broad points" in \cite{Guth} to the hyperbolic hyperboloid, motivating our definition through the geometry of the surface.  In Section 3, we use Kim's argument in \cite{Kim} to reduce Theorem \ref{thm1.1} to Theorem \ref{thm2.2}, an estimate on the contribution to $\fE_r$ from broad points.  Finally, in Section 4, the heart of the article, we prove Theorem \ref{thm2.2} using polynomial partitioning as in \cite{Guth}.

\medskip {\bf Notation and terminology.} As is standard, we write $A \lesssim B$ or $A = O(B)$ if there exists a constant $C > 0$ such that $A \leq CB$.  Generally, an implicit constant is not allowed to depend on any parameters present in the article.  In particular, constants never depend on the parabolic scaling parameter $r$.  There are exceptions: In Section 4, constants may depend on the exponent $\varepsilon$ from Theorems \ref{thm2.2} and \ref{thm2.4}.  To highlight dependence on a parameter $s$, we will sometimes write $\lesssim_s$ in place of $\lesssim$.  Likewise, we write $ c \ll 1$ to mean that $c$ is sufficiently small, and we use subscripts to indicate dependence on  parameters.   A number $\delta$ is dyadic if $\delta = 2^j$ for some $j \in \Z$, and an interval $I$ is dyadic if $I = [k2^j, (k+1)2^j)$ for some $j,k \in \Z$.  If $u,v$ are geometric objects that form an angle, then $\angle(u,v)$ denotes the measure of their angle.  Finally, ``hyperboloid" always means the hyperbolic (one-sheeted) hyperboloid.

\medskip
{\bf Acknowledgments.}  The author is very grateful to Betsy Stovall for her advice.  This project was suggested by Stovall and grew out of joint work with Stovall and Diogo Oliveira e Silva.   The 2019 MSRI Summer Graduate School on the Polynomial Method provided useful discussions during the earliest stage of this project.  The author was supported by NSF grant DMS-1653264.

\medskip
{\bf A revised note on publication.} An earlier version of this manuscript was to be combined with \cite{Guo--Oh}, with only the resulting joint manuscript of Bruce, Guo, and Oh being submitted for publication.  However, due to reasons unforeseen at that time, all authors have now mutually agreed to submit the two manuscripts separately.

\section{Broad points and the geometry of the hyperboloid}
In this section, we adapt the notion of ``broad points" to the hyperboloid.  Informally, given a function $f \in L^\infty(U)$, a point $(t,x) \in \R \times \R^2$ is ``broad" for $\fE_rf$ if there exist small, well-separated squares $\tau_1,\tau_2 \subseteq U$ such that $f\chi_{\tau_1}$ and $f\chi_{\tau_2}$ contribute significantly to $\fE_rf(t,x)$; otherwise $(t,x)$ is ``narrow."  To estimate $\fE_rf$, it suffices to bound the contributions from broad and narrow points separately.  The narrow contribution will be handled by a parabolic rescaling argument, since (morally) its Fourier transform is supported in a small rectangular cap in $\Sigma_r$.  The broad contribution will be handled by polynomial partitioning, using, in particular, some techniques from bilinear restriction theory.  In the latter argument, the precise separation condition imposed on the squares $\tau_1,\tau_2$ will be crucial for ensuring that their lifts to $\Sigma_r$ are appropriately transverse.  Our choice of this condition will be motivated by the geometry of the hyperboloid, which we now describe.

First, the basic symmetries of the hyperboloid are the Lorentz transformations, linear maps on $\R \times \R^2$ that preserve the quadratic form $(\tau,\xi) \mapsto \tau^2 - \xi_1^2 +\xi_2^2$.  Concretely, the spatial rotations
\begin{align}\label{rotations}
R_\omega(\tau,\xi) := (-\omega_2\xi_2+\omega_1\tau, \xi_1, \omega_1\xi_2+\omega_2\tau), \quad\quad \omega \in \mathbb{S}^1,
\end{align}
boosts
\begin{align}\label{boosts}
B_\nu(\tau,\xi) := \big(-\nu\xi_1+\sqrt{1+\nu^2}\tau, \sqrt{1+\nu^2}\xi_1-\nu\tau,\xi_2\big), \quad\quad \nu \in \R,
\end{align}
and dilations
\begin{align}\label{dilations}
D_\lambda(\tau,\xi) := \bigg(\tau, \frac{\lambda+\lambda^{-1}}{2}\xi_1 + \frac{\lambda-\lambda^{-1}}{2}\xi_2, \frac{\lambda-\lambda^{-1}}{2}\xi_1+\frac{\lambda+\lambda^{-1}}{2}\xi_2\bigg), \quad\quad \lambda \in \R,
\end{align}
will be of particular use to us.  We define a measure $d\mu$ on $\Sigma$ by setting
\begin{align}\label{invariant measure}
\int_\Sigma g d\mu := \int_\Omega g(\phi(\xi),\xi)\frac{d\xi}{\phi(\xi)}
\end{align}
for $g$ continuous and compactly supported.  This measure is Lorentz invariant in the following sense:  If $L$ is a Lorentz transformation and $\supp g \subseteq \Sigma$ and $L^{-1}(\supp g)\subseteq \Sigma$, then
\begin{align*}
\int_\Sigma (g \circ L)d\mu = \int_\Sigma g d\mu.
\end{align*}
We also record the following notation for later use.  Given a Lorentz transformation $L$ and $\xi \in \Omega$, let
\begin{align}\label{projection}
\overline{L}(\xi) := \pi(L(\phi(\xi),\xi)),
\end{align}
where $\pi(\tau,\xi) := \xi$ is the projection to the spatial coordinates.  If ${L}(\phi(\xi),\xi) \in \Sigma$ (equivalently, if $e_1\cdot L(\phi(\xi),\xi) \geq 0$), then $\overline{ML}(\xi) = \overline{M}(\overline{L}(\xi))$ for any other Lorentz transformation $M$.  In particular, if $V \subseteq \Omega$ and $L(\phi(\xi),\xi) \in \Sigma$ for $\xi \in V$, then $\overline{L}$ is invertible on $V$ with $\overline{L}^{-1}(\zeta) = \overline{L^{-1}}(\zeta)$ for $\zeta \in \overline{L}(V)$.  

Second, the (hyperbolic) hyperboloid is doubly ruled.   The aforementioned separation condition will be adapted to this structure:  Informally, two small squares $\tau_1,\tau_2 \subseteq U$ will be ``separated" if their lifts to the hyperboloid do not intersect a common ruling.  While the precise version of this condition will be stated in Section 4, we record a few preparatory details here.  The \emph{Lorentz norm} of $(\tau,\xi) \in \R \times \R^2$ is defined as
\begin{align*}
\llbracket(\tau,\xi)\rrbracket := \sqrt{|\tau^2 - \xi_1^2 + \xi_2^2|}.
\end{align*}
It is clearly Lorentz invariant, and if $(\tau,\xi),(\tau',\xi') \in \Sigma$, then $\llbracket(\tau,\xi)-(\tau',\xi')\rrbracket = 0$ if and only if $(\tau,\xi)$ and $(\tau',\xi')$ belong to the same ruling of $\Sigma$.  Indeed, the latter property can be checked by using the formulae
\begin{align}\label{rulings}
\ell_{(\tau,\xi)}^\pm(t) := (\tau,\xi) + t(\xi_1\tau \mp \xi_2, 1+\xi_1^2, \xi_1\xi_2 \pm \tau),
\end{align}
which parametrize the rulings $\ell_{(\tau,\xi)}^\pm \subset \Sigma$ that intersect at $(\tau,\xi) \in \Sigma$.  We also define the \emph{Lorentz separation} of $\xi,\zeta \in \Omega$ as the quantity
\begin{align*}
\operatorname{dist_L}(\xi,\zeta) := \llbracket (\phi(\xi),\xi) - (\phi(\zeta),\zeta)\rrbracket,
\end{align*}
which can be viewed as the ``distance" between $(\phi(\xi),\xi)$ and $(\phi(\zeta),\zeta)$ modulo the rulings of $\Sigma$.  Given this definition, a more accurate rendering of our separation condition would be that $\ldist(\xi,\zeta) \gtrsim 1$ for all $\xi \in \tau_1$ and $\zeta \in \tau_2$.  Near the end of this section, we will prove Lemma \ref{lem4.4*}, which relates $\ldist$ to other (genuine) distances.

Having described the geometry of the hyperboloid, we turn to defining broad points.   Our first step is to divide each surface $\Sigma_r$ into caps that lie above special sets which we call \emph{tiles}.  Consider the map $\Phi : \R^2 \rightarrow \R^2$ given by
\begin{align*}
\Phi(\xi) := \frac{(\xi_1\sqrt{1+\xi_2^2}+\xi_2\sqrt{1+\xi_1^2}, \xi_2-\xi_1)}{\sqrt{1+\xi_1^2}+\sqrt{1+\xi_2^2}}
\end{align*}
and, for each $r \in (0,1]$, let $\Phi_r(\xi) := r^{-1}\Phi(r\xi)$.  Recall the constant $\delta_0$  used to define $U$, and assume henceforth that $\delta_0$ is dyadic.  Given two dyadic numbers $ \delta,\delta' \in (0,\delta_0]$, a \emph{$(\delta,\delta',r)$-tile} is any nonempty set of the form
\begin{align*}
\rho := \Phi_r(I_\delta \times I_{\delta'}) \cap U,
\end{align*}
where $I_\delta$ and $I_{\delta'}$ are dyadic intervals contained in $[-\delta_0,\delta_0)$ of length $\delta$ and $\delta'$, respectively.   We denote the set of $(\delta, \delta',r)$-tiles by $\fT_{\delta,\delta',r}$.  Observe that $\Phi$ is a diffeomorphism near the origin.  (Indeed, $\Phi$ can be viewed as a perturbation of the rotation $\xi \mapsto \frac{1}{2}(\xi_1+\xi_2, \xi_2-\xi_1)$ for $\xi$ small.)  Taking $\delta_0$ sufficiently small, it is straightforward to check that $\|\Phi_r^{-1}\|_{C^1(U)} \leq 5$ uniformly in $r$, and consequently that $U \subseteq \Phi_r([-\delta_0,\delta_0)^2)$ for every $r$.  We also note that for fixed $\delta,\delta',r$, the $(\delta,\delta',r)$-tiles are pairwise disjoint and satisfy
\begin{align*}
U = \bigcup_{\rho \in \fT_{\delta,\delta',r}}\rho.
\end{align*}
Let us briefly mention the geometry underlying these definitions.  The map $\Phi$ was created with the following property in mind:  If $\ell \subset \R^2$ is a vertical or horizontal line that intersects $\Phi_r^{-1}(U)$, then $\Phi_r(\ell)$ is a line that lifts to a ruling in $\Sigma_r$.  Thus, each tile lifts to a quadrilateral (in fact, nearly rectangular) cap bounded by four rulings.  We can think of the collection $\{\fT_{\delta,\delta',r}\}_{\delta,\delta'}$ as a dyadic grid adapted to $\Sigma_r$.  A precise description of the geometry of $\Phi$ will appear in Lemma \ref{lem2.3} at the end of this section. 

Now, let $K \geq \delta_0^{-1}$ be a large dyadic constant.  As suggested above, we will analyze contributions to $\fE_r$ from square-like sets $\tau$.  The $(K^{-1},K^{-1},r)$-tiles will function as these basic pieces.  However, controlling contributions from longer rectangle-like sets will also be essential. (As we will see, a collection of non-separated squares $\tau$ must cluster around a line.) For each dyadic number $\delta \in [K^{-1},\delta_0]$, let
\begin{align*}
\fR_{\delta,r} := \fT_{K^{-1},\delta,r} \cup \fT_{\delta,K^{-1},r}
\end{align*}
and also set
\begin{align*}
\fR_r := \bigcup_{\delta \in [K^{-1},\delta_0]} \fR_{\delta,r}.
\end{align*}
Elements of $\fR_{\delta,r}$ are nearly rectangles of dimensions $K^{-1} \times \delta$ and slope approximately $1$ or $-1$.  We are now ready to define broad points.  Given $f \in L^\infty(U)$ and $\alpha \in (0,1]$, we say that $(t,x) \in \R \times \R^2$ is \emph{$\alpha$-broad for $\fE_rf$} if
\begin{align*}
\max_{\rho \in \fR_r}|\fE_rf_\rho(t,x)| \leq \alpha|\fE_rf(t,x)|,
\end{align*}
where $f_\rho := f\chi_\rho$.  The \emph{$\alpha$-broad part} of $\fE_rf$ is defined as
\begin{align*}
\Br_\alpha\fE_rf(t,x) := \begin{cases} \fE_rf(t,x) &\text{if $(t,x)$ is $\alpha$-broad for } \fE_rf,\\ 0 &\text{otherwise}.\end{cases}
\end{align*}
In the next section, we will reduce Theorem \ref{thm1.1} to the following estimate on the broad part:

\begin{theorem}\label{thm2.2}
For every $0 < \varepsilon \ll 1$, there exists a constant $C_\varepsilon$, depending only on $\varepsilon$, such that if $K = 2^{\lceil \varepsilon^{-10}\rceil}$, then
\begin{align*}
\|\Br_{K^{-\varepsilon}}\fE_rf\|_{L^{13/4}(B_R)} \leq C_\varepsilon R^\varepsilon\|f\|_2^{12/13}\|f\|_\infty^{1/13}
\end{align*}
for all $ r \in (0,1]$, $R \geq 1$, and balls $B_R$ of radius $R$.
\end{theorem}

To conclude this section, we present two geometric lemmas.  We will need the following notation:  For $\xi \in \Omega$, let $\ell_\xi^\pm$ denote the lines in $\R^2$ parametrized by
\begin{align}\label{projections}
\ell_\xi^\pm(t) := \xi + t(1 + \xi_1^2, \xi_1\xi_2 \pm \phi(\xi)).
\end{align}
Geometrically, $\ell_\xi^\pm$ are the projections to the spatial coordinates of the lines $\ell_{(\phi(\xi),\xi)}^\pm$ defined in \eqref{rulings}.  

\begin{lemma}\label{lem4.4*}
For all $\xi,\zeta \in U$, we have
\begin{itemize}
\item[\textnormal{(a)}]{$\dist(\xi, \ell_\zeta^+ \cup \ell_\zeta^-) \lesssim  \ldist(\xi,\zeta) \lesssim |\xi-\zeta|$;}
\item[\textnormal{(b)}]{$\ldist(\xi,\zeta)^2 \sim |\langle(\nabla^2\phi(\xi))^{-1}(\nabla\phi(\xi)-\nabla\phi(\zeta)),\nabla\phi(\xi)-\nabla\phi(\zeta)\rangle|$.}
\end{itemize}
\end{lemma}
\begin{proof}
(a) Let $\xi'$ be the intersection of $\ell_\xi^-$ and $\ell_\zeta^+$.  An easy calculation shows that $\angle(\ell_\eta^+,\ell_{\eta'}^-) \gtrsim 1$ for all $\eta,\eta' \in U$.  (In fact, the lines are nearly orthogonal.)  In particular, the law of sines implies that $\xi' \in CU$ for some constant $C$.  Let $L := B_\nu R_\omega$, as defined in \eqref{rotations} and \eqref{boosts}, with
\begin{align*}
\nu &:= \zeta_1,\\
\omega &:= \bigg(\frac{\phi(\zeta)}{\sqrt{1+\zeta_1^2}}, -\frac{\zeta_2}{\sqrt{1+\zeta_1^2}}\bigg).
\end{align*} 
Then $L(\phi(\zeta),\zeta) = (1,0)$.  Let $\eta := \overline{L}(\xi)$ and $\eta' := \overline{L}(\xi')$, using the notation from \eqref{projection}.  Since $\nu$ and $\omega_2$ are very small, $\overline{L}$ is essentially a perturbation of the identity.  It is easy to check that $L(\phi(\xi),\xi) \in \Sigma$ for all $\xi \in CU$, provided $\delta_0$ is sufficiently small, and thus $\overline{L}$ is invertible on $CU$.  Additionally, we have the bound $\|\overline{L}^{-1}\|_{C^1(\overline{L}(CU))} \lesssim 1$.  Combining these facts, we see that
\begin{align}\label{ex2.7*}
\dist(\xi,\ell_\zeta^+) \leq |\xi - \xi'| \lesssim |\eta-\eta'|.
\end{align}
Since $L$ preserves the hyperboloid and is linear, it must permute the rulings of the surface.  Therefore, since $L$ is close to the identity,
\begin{align*}
\overline{L}(\ell_\zeta^+) &= \ell_0^+ = \R(1,1),\\
\overline{L}(\ell_\xi^-) &= \ell_\eta^-,
\end{align*}
which implies that $\{\eta'\} = \R(1,1) \cap \ell_\eta^-$.  Then, since $\eta \in \ell_\eta^-$ and $\angle(\R(1,1),\ell_\eta^-) \gtrsim 1$, it follows that $|\eta -\eta'| \lesssim \dist(\eta,\R(1,1))$.  Thus, by \eqref{ex2.7*}, we have $\dist(\xi,\ell_\zeta^+) \lesssim \dist(\eta,\R(1,1)) \leq |\eta_1-\eta_2|$.  A similar argument shows that $\dist(\xi,\ell_\zeta^-) \lesssim |\eta_1+\eta_2|$.  Hence,
\begin{align*}
\dist(\xi,\ell_\zeta^+ \cup \ell_\zeta^-) &= \min\{\dist(\xi,\ell_\zeta^+),\dist(\xi,\ell_\zeta^-)\}\\
&\leq \sqrt{|\eta_1^2 - \eta_2^2|}\\
&\lesssim \sqrt{|1-\phi(\eta)|}\\
&\sim\llbracket (\phi(\eta), \eta) - (1,0)\rrbracket\\
&= \ldist(\xi,\zeta),
\end{align*}
where the last step used the Lorentz invariance of the Lorentz norm.  The second inequality in (a) can be proved in a similar (but easier) fashion.  (It also follows from part (b), using the Cauchy--Schwarz inequality and bounds on the derivatives of $\phi$.)

(b) A straightforward computation shows that the right-hand side of (b) is
\begin{align*}
\frac{1}{\phi(\xi)\phi(\zeta)^2}|(1+\xi_1^2)(\xi_1\phi(\zeta) - \zeta_1\phi(\xi))^2 + 2\xi_1\xi_2(-\xi_2\phi(\zeta) &+ \zeta_2\phi(\xi))(\xi_1\phi(\zeta) - \zeta_1\phi(\xi))\\ &\hspace*{.23 in}+ (-1+\xi_2^2)(-\xi_2\phi(\zeta) + \zeta_2\phi(\xi))^2|.
\end{align*}
The expression inside absolute value signs is equal to
\begin{align*}
\phi(\xi)^2[(1+\xi_1^2)\zeta_1^2 - 2\xi_1\xi_2\zeta_1\zeta_2+(-1+\xi_2^2)\zeta_2^2]\\
 &\hspace{-1.1 in}+ \phi(\xi)\phi(\zeta)[-2(1+\xi_1^2)\xi_1\zeta_1 + 2\xi_1\xi_2(\xi_2\zeta_1+\xi_1\zeta_2)-2(-1+\xi_2^2)\xi_2\zeta_2]\\ &\hspace{-1.1 in}+ \phi(\zeta)^2[(1+\xi_1^2)\xi_1^2 - 2\xi_1^2\xi_2^2 + (-1+\xi_2^2)\xi_2^2],
\end{align*}
which, by the relations $\phi(\xi)^2 = 1 + \xi_1^2 - \xi_2^2$ and $\phi(\zeta)^2 = 1 +\zeta_1^2 - \zeta_2^2$, simplifies to
\begin{align*}
\phi(\xi)^2[(\xi_1\zeta_1-\xi_2\zeta_2)^2+\phi(\zeta)^2-1] + 2\phi(\xi)\phi(\zeta)[-\xi_1\zeta_1+\xi_2\zeta_2] + \phi(\zeta)^2\phi(\xi)^2[\phi(\xi)^2-1].
\end{align*}
Thus, by a bit more algebra, the right-hand side of (b) is
\begin{align*}
\frac{\phi(\xi)}{\phi(\zeta)^2}|1+\xi_1\zeta_1-\xi_2\zeta_2-\phi(\xi)\phi(\zeta)||\xi_1\zeta_1-\xi_2\zeta_2-\phi(\xi)\phi(\zeta)-1|.
\end{align*}
We also compute that
\begin{align*}
\ldist(\xi,\zeta)^2 = 2|1+\xi_1\zeta_1-\xi_2\zeta_2-\phi(\xi)\phi(\zeta)|,
\end{align*}
and (b) follows.
\end{proof}

Let us briefly interpret the lemma.  Part (a) says that points with small Lorentz separation lie near a common line, while points with large Lorentz separation are genuinely separated.  Part (b) relates to a transversality condition that naturally arises in bilinear restriction theory (see \cite[Theorem 1.1]{Lee}).  Crucially, whenever $\xi$ and $\zeta$ belong to separated squares (as discussed above), the right-hand side of (b) will be bounded below.

\begin{lemma}\label{lem2.3}
If $\xi \in U$ and  $\zeta = \Phi^{-1}(\xi)$, then $\ell_\xi^+ \cap 2U = \Phi(\{\zeta_1\} \times \R) \cap 2U$ and $\ell_\xi^- \cap 2U = \Phi(\R \times \{\zeta_2\}) \cap 2U$.
\end{lemma}
\begin{proof}
We will only prove the first equality; the second follows in a similar manner.  The proof rests on two claims.

\emph{Claim 1. If $|\xi|,|\xi'| \leq 1/2$ and $\xi' \in \ell_\xi^+$, then $\ell_{\xi'}^+ = \ell_\xi^+$.}  Consider the lines $\ell_\xi^+$, $\ell_{\xi'}^+$, and $\ell_{\xi'}^-$.  Each one contains $\xi'$ and lifts to a ruling in $\Sigma$.  By elementary geometry, no three rulings of the hyperboloid intersect at a common point.  Thus, two of these lines are identical.  Since $\ell_{\xi'}^+ \neq \ell_{\xi'}^-$ and $\ell_{\xi}^+ \neq \ell_{\xi'}^-$, as is easy to check, we must have $\ell_{\xi'}^+ = \ell_\xi^+$.

\emph{Claim 2. For every $\xi \in \R^2$, we have $\{\Phi(\xi)\} = \ell_{(\xi_1,0)}^+ \cap \ell_{(\xi_2,0)}^-$.}  This relation can be checked directly, using \eqref{projections}.  It is helpful to reparametrize \eqref{projections} so that the second coordinates of $\ell_{(\xi_1,0)}^+(t)$ and $\ell_{(\xi_2,0)}^-(t)$ are identically $t$ and $-t$, respectively.

Now, fix $\xi \in U$ and let $\zeta := \Phi^{-1}(\xi)$.  Let $\xi' \in \ell_\xi^+ \cap 2U$ and $\zeta' := \Phi^{-1}(\xi')$.  Claim 2 implies that $\xi \in \ell_{(\zeta_1,0)}^+$ and $\xi' \in \ell_{(\zeta_1',0)}^+$.  Hence, by claim 1, we have $\ell_{(\zeta_1,0)}^+ = \ell_\xi^+ = \ell_{\xi'}^+ = \ell_{(\zeta_1',0)}^+$, and it follows that $\zeta_1 = \zeta_1'$.  Since $\xi'$ was arbitrary, we conclude that $\ell_\xi^+ \cap 2U \subseteq \Phi(\{\zeta_1\} \times \R) \cap U$.  The other direction is similar:  Let $\xi' \in \Phi(\{\zeta_1\} \times \R) \cap 2U$, so that $\xi' = \Phi(\zeta_1,t)$ with $(\zeta_1,t) \in \Phi^{-1}(2U) \subseteq \Omega$.  Claim 2 implies that $\xi,\xi' \in \ell_{(\zeta_1,0)}^+$.  Hence, $\xi' \in \ell_\xi^+$ by claim 1, and it follows that $\Phi(\{\zeta_1\} \times \R) \cap 2U \subseteq \ell_\xi^+ \cap 2U$.
\end{proof}

\section{Reduction to Theorem \ref{thm2.2}}

In this section, we adapt the argument of Kim in \cite{Kim} to show that Theorem \ref{thm2.2} implies Theorem \ref{thm1.1}.  The following parabolic rescaling lemma is the main tool required for this reduction.

\begin{lemma}\label{lem5.1}
Let $r \in (0,1]$ be dyadic and let $\theta \in [0,1]$.  If $\|\fE_s g\|_{L^q(B_{R/2})} \leq M\|g\|_2^{1-\theta}\|g\|_\infty^\theta$ for all $s \in (0,1]$, balls $B_{R/2}$ of radius $R/2$, and $g \in L^\infty(U)$, then there exists an absolute constant $C$ such that $\|\fE_r h\|_{L^q(B_R)} \leq CM(\delta\delta')^{\frac{1+\theta}{2} -\frac{2}{q}}\|h\|_2^{1-\theta}\|h\|_\infty^\theta$ for all bounded $h$ supported in $\rho \in \fT_{\delta,\delta',r}$, provided $\delta,\delta'$ are sufficiently small.
\end{lemma}
\begin{proof}
Fix $h \in L^\infty(U)$ supported in $\rho \in \fT_{\delta,\delta',r}$.  There exists $\rho_1 \in \fT_{r\delta,r\delta',1}$ such that $r\rho \subseteq \rho_1$.  By parabolic rescaling, we have
\begin{align*}
\|\fE_r h\|_{L^q(B_R)} = r^{\frac{4}{q}-2}\|\fE_1h_{\rho_1}\|_{L^q(P_r(B_R))},
\end{align*}
where $h_{\rho_1} := h(r^{-1}\cdot)$ is supported in $\rho_1$.  We assume without loss of generality that $\delta \leq \delta'$ and fix $\eta \in \rho_1$.  We claim that $\rho_1$ lies in the intersection of an $O(r\delta)$-neighborhood of $\ell_\eta^+$ and an $O(r\delta')$-neighborhood of $\ell_\eta^-$.  Indeed, let $\eta' \in \rho_1$ and set $\zeta = \Phi^{-1}(\eta)$ and $\zeta' = \Phi^{-1}(\eta')$.  By the definition of $(r\delta,r\delta',1)$-tile, we have
\begin{align*}
\dist(\zeta', (\{\zeta_1\}\times\R) \cap \Phi^{-1}(U)) &\leq r\delta,\\
\dist(\zeta', (\R\times\{\zeta_2\}) \cap \Phi^{-1}(U)) &\leq r\delta'.
\end{align*}
Thus, by Lemma \ref{lem2.3} and the boundedness of $\|\nabla \Phi\|$ near the origin, it follows that
\begin{align*}
\dist(\eta', \ell_\eta^+) &\lesssim  r\delta,\\
\dist(\eta', \ell_\eta^-) &\lesssim  r\delta',
\end{align*}
proving the claim.

Now, let $L := (D_\lambda B_\nu R_\omega)^{-1}$ with
\begin{align*}
\lambda &:= \sqrt{\frac{\delta}{\delta'}},\\
\nu &:= \eta_1,\\
\omega &:= \bigg(\frac{\phi(\eta)}{\sqrt{1+\eta_1^2}}, -\frac{\eta_2}{\sqrt{1+\eta_1^2}}\bigg),
\end{align*}
using the notation from \eqref{rotations}--\eqref{dilations}.  As in the proof of Lemma \ref{lem4.4*}, the map $\overline{B_\nu R_\omega}$ sends $\eta$ to the origin and $\ell_\eta^\pm$ to $\ell_0^\pm = \R(1,\pm1)$ and satisfies $\|\overline{B_\nu R_\omega}\|_{C^1(U)} \lesssim 1$.  Thus, by the claim, $\overline{B_\nu R_\omega}(\rho_1)$ lies in an $O(r\delta) \times O(r\delta')$ rectangle with slope $1$ centered at the origin, and consequently $\overline{D_\lambda}(\overline{B_\nu R_\omega}(\rho_1))$ is contained in $sU$ for some $s \lesssim r\sqrt{\delta\delta'}$.  It is easy to check that $B_\nu R_\omega(\phi(\xi),\xi) \in \Sigma$ for all $\xi \in U$, and thus by the discussion following \eqref{projection},
\begin{align}\label{ex3.1*}
\overline{L^{-1}}(\rho_1) = \overline{D_\lambda}(\overline{B_\nu R_\omega}(\rho_1)) \subseteq sU.
\end{align}
We claim that
\begin{align}\label{ex3.2*}
\overline{L}^{-1}(\rho_1) := \{\xi \in \Omega : \overline{L}(\xi) \in \rho_1\} = \overline{L^{-1}}(\rho_1).
\end{align}
Indeed, given a set $V \subseteq \Omega$, let $V^\pm := \{(\pm \phi(\xi),\xi) : \xi \in V\}$.  Then
\begin{align*}
\overline{L}^{-1}(\rho_1) &= \{\xi \in \Omega : L(\phi(\xi),\xi) \in \rho_1^+ \cup \rho_1^-\}\\ &= \{\xi \in \Omega : (\phi(\xi),\xi) \in L^{-1}(\rho_1^+) \cup -L^{-1}((-\rho_1)^+)\}.
\end{align*}
It is easy to check that $e_1 \cdot L^{-1}(\phi(\zeta),\zeta) > 0$ for every $\zeta \in U$.  Thus, since $-\rho_1 \subseteq U$ and $\phi \geq 0$, we have $(\phi(\xi),\xi) \notin -L^{-1}((-\rho_1)^+)$ for every $\xi$.  Hence,
\begin{align*}
\overline{L}^{-1}(\rho_1) = \{\xi \in \Omega : (\phi(\xi),\xi) \in L^{-1}(\rho_1^+)\} = \overline{L^{-1}}(\rho_1),
\end{align*}
proving the claim.

Now, define $F : \Sigma \rightarrow \C$ by $F(\tau,\xi) := h_{\rho_1}(\xi)\phi(\xi)$ and assume that $\delta,\delta'$ are small enough that $s \leq 1$.  Then, using \eqref{ex3.2*} and \eqref{ex3.1*}, it is straightforward to check that $L^{-1}(\supp F) \subseteq \Sigma$.  Thus,
\begin{align*}
\fE_1h_{\rho_1}(t,x) &= e^{-2\pi it}\int_\Sigma e^{2\pi i(t,x)\cdot(\tau,\xi)}F(\tau,\xi)d\mu(\tau,\xi)\\
&=e^{-2\pi it}\int_\Sigma e^{2\pi i(t,x)\cdot L(\tau,\xi)}F(L(\tau,\xi))d\mu(\tau,\xi),
\end{align*}
where $d\mu$ is the Lorentz-invariant measure given by \eqref{invariant measure}.  Hence, for $H(\xi) := h_{\rho_1}(\overline{L}(\xi))\frac{\phi(\overline{L}(\xi))}{\phi(\xi)}$, we have
\begin{align*}
|\fE_1h_{\rho_1}(t,x)| = |\fE_1H(L^*(t,x))|.
\end{align*}
Noting that $|\det L| = 1$, we obtain the relation
\begin{align*}
\|\fE_r h\|_{L^q(B_R)} = r^{\frac{4}{q}-2}\|\fE_1 H\|_{L^q(L^* P_r(B_R))},
\end{align*}
and parabolic rescaling then gives
\begin{align*}
\|\fE_rh\|_{L^q(B_R)} \sim (\delta\delta')^{1-\frac{2}{q}}\|\fE_s[H(s\cdot)]\|_{L^q(P_{s^{-1}}L^* P_r(B_R))},
\end{align*}
where $H(s\cdot)$ is supported in $U$ by \eqref{ex3.2*} and \eqref{ex3.1*}.

We claim that $P_{s^{-1}}L^* P_r(B_R)$ is covered by a bounded number of balls of radius $R/2$.  Assuming the claim is true, the hypothesis of the lemma implies that
\begin{align}\label{ex3.3*}
\|\fE_r h\|_{L^q(B_R)} \lesssim M(\delta\delta')^{1-\frac{2}{q}}\|H(s\cdot)\|_2^{1-\theta}\|H(s\cdot)\|_\infty^\theta.
\end{align}
To prove the claim, we may assume by translation invariance that $B_R$ is centered at the origin.  Let $Q(a,b,c)$ denote any rectangular box centered at zero with sides of length $O(a),O(b),O(c)$ parallel to $(1,0,0)$, $(0,1,1)$, $(0,1,-1)$, respectively.  Thus, slightly informally, $B_R \subseteq Q(R,R,R)$, and
\begin{align*}
P_r(B_R) \subseteq Q\bigg(\frac{R}{r^2},\frac{R}{r},\frac{R}{r}\bigg).
\end{align*}
We have $L^* = D_\lambda^{-*}B_\nu^{-*}R_\omega^{-*}$, where $S^{-*} := (S^{-1})^*$.  Since $R_\omega^{-*}$ and $B_\nu^{-*}$ have bounded norm, we can ignore their contribution.  Thus, from the definition of $D_\lambda$, we have
\begin{align*}
L^* P_r(B_R) \subseteq Q\bigg(\frac{R}{r^2},\frac{\sqrt{\delta'}R}{\sqrt{\delta}r},\frac{\sqrt{\delta}R}{\sqrt{\delta'}r}\bigg).
\end{align*}
The definition of $s$ then implies that
\begin{align*}
P_{s^{-1}}L^* P_r(B_R) \subseteq Q(\delta\delta' R, \delta'R, \delta R),
\end{align*}
which proves claim.

Finally, to finish the proof, we need to undo the changes of variable we have used.  Using \eqref{ex3.2*} and \eqref{ex3.1*}, we have $L(\phi(\xi),\xi) \in \Sigma$ for all $\xi \in \overline{L}^{-1}(\rho_1)$.  Thus, $\overline{L}$ is invertible on $\supp H$ with $\overline{L}^{-1}(\zeta) = \overline{L^{-1}}(\zeta)$ for $\zeta \in \overline{L}(\supp H) \subseteq U$.  Moreover, $\overline{L^{-1}} = \overline{D_\lambda}\circ \overline{B_\nu} \circ \overline{R_\omega}$ on $U$, so a straightforward calculation shows that $|\det \nabla \overline{L}^{-1}| \lesssim 1$ on $U$.  Using these observations, we find that
\begin{align*}
\|H(s\cdot)\|_2 &\lesssim \frac{1}{\sqrt{\delta\delta'}}\|h\|_2,\\
\|H(s\cdot)\|_\infty &\lesssim \|h\|_\infty.
\end{align*}
Plugging these bounds into \eqref{ex3.3*} completes the proof. 
\end{proof}

\begin{proposition}\label{prop5.2}
Assume that Theorem \ref{thm2.2} holds.  Then for every $\theta \in (3/13, 1]$ and $0 < \varepsilon \ll_\theta 1$, there exists a constant $C_{\varepsilon,\theta}$, depending only on $\varepsilon$ and $\theta$, such that
\begin{align*}
\|\fE_r f\|_{L^{13/4}(B_R)} \leq C_{\varepsilon,\theta}R^\varepsilon\|f\|_2^{1-\theta}\|f\|_\infty^\theta
\end{align*}
for all $r \in (0,1]$, $R \geq 1$, and balls $B_R$ of radius $R$.
\end{proposition}
\begin{proof}
We will induct on $R$.  The base case, that $R \sim 1$, holds trivially.  We assume as our induction hypothesis that the proposition holds with $R/2$ in place of $R$.  Additionally, we may assume that $2C_\varepsilon^{13/4} \leq C_{\varepsilon,\theta}^{13/4}$, where $C_\varepsilon$ is the constant from Theorem \ref{thm2.2}.  The definition of $K^{-\varepsilon}$-broad implies that
\begin{align*}
|\fE_r f(t,x)| \leq \max\{|\Br_{K^{-\varepsilon}}\fE_rf(t,x)|,K^\varepsilon\max_{\rho \in \fR_r}|\fE_rf_\rho(t,x)|\}
\end{align*}
for every $(t,x) \in \R \times \R^2$.  It follows that
\begin{align*}
\int_{B_R}|\fE_rf|^{13/4} \leq \int_{B_R}|\Br_{K^{-\varepsilon}}\fE_rf|^{13/4} + K^{\frac{13}{4}\varepsilon}\sum_{\rho \in \fR_r}\int_{B_R}|\fE_rf_\rho|^{13/4} =: \textrm{I} + \textrm{II}.
\end{align*}
To bound the first term, we use Theorem \ref{thm2.2} and H\"older's inequality to get
\begin{align*}
\textrm{I} \leq (C_\varepsilon\|f\|_2^{12/13}\|f\|_\infty^{1/13})^{13/4} \leq (C_\varepsilon\|f\|_2^{1-\theta}\|f\|_\infty^\theta)^{13/4} \leq \frac{1}{2}(C_{\varepsilon,\theta}\|f\|_2^{1-\theta}\|f\|_\infty^{\theta})^{13/4}.
\end{align*}
To bound the second term, we will use Lemma \ref{lem5.1}.  We may assume that $r$ is dyadic by parabolic rescaling, and the other hypothesis of the lemma holds by our inductive assumption.  
Additionally, by H\"older's inequality, we may assume that $\theta$ is close to $3/13$; in particular, that $\theta \leq 5/13$.  Then,
\begin{align*}
\textrm{II} &\leq K^{\frac{13}{4}\varepsilon}\sum_{\delta \in [K^{-1},\delta_0]}\sum_{\rho \in \fR_{\delta,r}}(CC_{\varepsilon,\theta}R^{\varepsilon}(\delta K^{-1})^{\frac{1}{2}(\theta-\frac{3}{13})}\|f_\rho\|_{2}^{1-\theta}\|f_\rho\|_{\infty}^\theta)^{13/4}\\
&\leq K^{\frac{13}{4}\varepsilon + \frac{13}{8}(\frac{3}{13}-\theta)}C^{13/4}\sum_{\delta \in [K^{-1},\delta_0]}\delta^{\frac{13}{8}(\theta-\frac{3}{13})}(C_{\varepsilon,\theta}R^\varepsilon)^{13/4}\sum_{\rho \in \fR_{\delta,r}}\|f_\rho\|_{2}^{\frac{13}{4}(1-\theta)}\|f\|_\infty^{\frac{13}{4}\theta}\\
&\leq \bigg[K^{\frac{13}{4}\varepsilon + \frac{13}{8}(\frac{3}{13}-\theta)}C^{13/4}\bigg(\sum_{\delta \in [K^{-1},\delta_0]}\delta^{\frac{13}{8}(\theta-\frac{3}{13})}\bigg)2^{\frac{13}{8}(1-\theta)}\bigg](C_{\varepsilon,\theta}R^\varepsilon\|f\|_2^{1-\theta}\|f\|_\infty^\theta)^{13/4},
\end{align*}
where the last step used the inclusion $\ell^2 \hookrightarrow \ell^{\frac{13}{4}(1-\theta)}$ and that $\fR_{\delta,r}$ covers $U$ with overlap of multiplicity $2$.  Since $\theta > 3/13$, the sum over $\delta$ is bounded and the power of $K$ is negative for $\varepsilon$ sufficiently small.  Thus, since $K \rightarrow \infty$ as $\varepsilon \rightarrow 0$ by the hypothesis of Theorem \ref{thm2.2}, the expression in square brackets is at most $1/2$ for $\varepsilon$ sufficiently small, and the induction closes. 
\end{proof}

Assuming Theorem \ref{thm2.2} holds, Proposition \ref{prop5.2} implies the restricted strong type bounds
\begin{align*}
\|\fE_rf_E\|_{L^{13/4}(B_R)} \lesssim_{\varepsilon,p} R^\varepsilon|E|^{1/p}.
\end{align*}
for all $p > 13/5$, measurable sets $E \subseteq U$, and $|f_E| \lesssim \chi_E$.  Then, by real interpolation with the trivial $L^1 \rightarrow L^\infty$ estimate, we obtain the strong type bounds
\begin{align*}
\|\fE_r f\|_{L^{q}(B_R)} \lesssim_{\varepsilon,p,q}R^\varepsilon\|f\|_p
\end{align*}
for all $q > 13/4$ and $p > (q/2)'$.  Tao's epsilon removal lemma, in the form of Theorem 5.3 in \cite{Kim}, consequently gives the global strong type bounds
\begin{align}\label{off-scaling}
\|\fE_r f\|_q \lesssim_{p,q}\|f\|_p
\end{align}
for the same range of $p,q$, completing the proof of Theorem \ref{thm1.1}.

\section{Proof of Theorem \ref{thm2.2}}

We are left to prove Theorem \ref{thm2.2}, which will occupy the rest of the article.  To enable an inductive argument, we will actually need to prove a slightly stronger theorem, as in \cite{Guth}.  In Section 2, we defined broad points by considering the contribution to $\fE_rf$ from each $f_\rho$, where $f_\rho := f\chi_\rho$ and $\rho \in \fR_r$.  Soon we will work with wave packets of the form $\fE_rf_{\rho,T}$, where $f_{\rho,T}$ is supported not in $\rho$ but  in a slight enlargement of it.  Thus, we need a more general definition of broad points in which the functions $f_\rho$ may have larger, overlapping supports.  Given $\rho = \Phi_r(I_\delta \times I_{\delta'}) \cap U \in \fT_{\delta,\delta',r}$ and $m \geq 1$, we define
\begin{align*}
m\rho := \Phi_r(m(I_\delta \times I_{\delta'})) \cap U,
\end{align*}
where $m(I_\delta \times I_{\delta'})$ is the $m$-fold dilate of the rectangle $I_\delta \times I_{\delta'}$ with respect to its center.  Let
\begin{align*}
\fS_r := \fR_{K^{-1},r};
\end{align*}
elements of $\fS_r$ are essentially $K^{-1} \times K^{-1}$ squares.  Now, given $f \in L^2(U)$, suppose that $f = \sum_{\tau \in \fS_r} f_\tau$, with each $f_\tau$ supported in $m \tau$, for some $m \geq 1$.  In our modified definition, $(t,x) \in \R \times \R^2$ is \emph{$\alpha$-broad for $\fE_r f$} if
\begin{align*}
\max_{\rho \in \fR_r}\bigg|\sum_{\tau \in \fS_r : \tau \subseteq \rho}\fE_rf_\tau(t,x)\bigg| \leq \alpha|\fE_rf(t,x)|.
\end{align*}
We define the \emph{$\alpha$-broad part} of $\fE_rf$, still denoted by $\Br_{\alpha}\fE_rf$, as in Section 2.  These definitions depend on the particular decomposition $f = \sum_\tau f_\tau$.

\begin{theorem}\label{thm2.4}
For every $0 < \varepsilon \ll 1$, there exists a constant $C_\varepsilon'$, depending only on $\varepsilon$, such that if $K = 2^{\lceil \varepsilon^{-10}\rceil}$, then the following holds: 
If $f = \sum_{\tau \in \fS_r} f_\tau$ with each $f_\tau$ supported in $m \tau$, for some $m \geq 1$, and if additionally $f$ satisfies
\begin{align}\label{ex2.1}
\dashint_{D(\xi,R^{-1/2})}|f_\tau|^2 \leq 1
\end{align}
for all $\xi \in U$ and $\tau \in \fS_r$, then
\begin{align*}
\int_{B_R}|\Br_{\alpha}\fE_rf|^{13/4} \leq C_\varepsilon'R^{\varepsilon+\varepsilon^6\log(K^\varepsilon\alpha m^2)}\bigg(\sum_{\tau \in \fS_r}\|f_\tau\|_2^2\bigg)^{3/2+\varepsilon}
\end{align*}
for all $r \in (0,1]$,  $R \gg_\varepsilon 1$, balls $B_R$ of radius $R$, and $\alpha \in [K^{-\varepsilon}, 1]$.
\end{theorem}
 
 A couple of remarks may be helpful.  Firstly, the dyadic structure of our tiles, as defined in Section 2, implies that if $\tau \in \fS_r$ and $\rho \in \fR_r$, then either $\tau \cap \rho = \emptyset$ or $\tau \subseteq \rho$.  More generally, if $\rho_1 \in \fT_{\delta_1,\delta_1',r}$ and $\rho_2 \in \fT_{\delta_2,\delta_2',r}$, then either $\rho_1 \cap \rho_2 = \emptyset$ or $\rho_1 \cap \rho_2 \in \fT_{\min_i \delta_i, \min_i\delta_i',r}$.  Secondly, Theorem \ref{thm2.4} is indeed stronger than Theorem \ref{thm2.2}.  We can derive the latter from the former as follows:  If $R \sim_\varepsilon 1$, then the estimate in Theorem \ref{thm2.2} is trivial, so we may assume that $R \gg_\varepsilon 1$.  By scaling, we also may assume that $\|f\|_\infty = 1$.  Thus, the condition \eqref{ex2.1} holds automatically.  We now apply Theorem \ref{thm2.4} with $\alpha = K^{-\varepsilon}$ and $m=1$ to get
\begin{align*}
\int_{B_R}|\Br_{K^{-\varepsilon}}\fE_rf|^{13/4} \leq C_\varepsilon' R^{\varepsilon}\|f\|_2^{3+2\varepsilon} \leq |U|^{\varepsilon}C_\varepsilon' R^{\varepsilon}\|f\|_2^3,
\end{align*}
and then raising both sides to the power ${4/13}$ finishes the proof.

\subsection{Preliminaries}
Before beginning the proof of Theorem \ref{thm2.4}, we lay some groundwork.   For the remainder of the article, $\varepsilon$, $m$, $r$, $R$, $B_R$, and $\alpha$ are fixed.  Implicit constants will be allowed to depend on $\varepsilon$.  The propositions and lemma we record in this subsection are by now standard. 

We begin with the wave packet decomposition.  Let $\Theta$ be a collection of discs $\theta$ of radius $R^{-1/2}$ which cover $U$ with bounded overlap.  We denote by $c_\theta$ the center of $\theta$, and we let $v_\theta$ be the unit normal vector to $\Sigma_r$ at $(\phi_r(c_\theta),c_\theta)$.  We may assume that $c_\theta \in U$ for every $\theta$. Let $\delta := \varepsilon^2$, and for each $\theta$, let $\bbT(\theta)$ be a collection of tubes parallel to $v_\theta$ with radius $R^{\delta+1/2}$ and length $R$ and which cover $B_R$ with bounded overlap.  If $T \in \bbT(\theta)$, then $v(T) := v_\theta$ denotes the direction of $T$.  Finally, we set $\bbT := \bigcup_{\theta \in \Theta}\bbT(\theta)$.  The following wave packet decomposition resembles Proposition 2.6 in \cite{Guth}:

\begin{proposition}\label{prop3.1}
For each $T \in \bbT$, there exists a function $f_T \in L^2(\R^2)$ such that:
\begin{itemize}
\item[\textnormal{(i)}]{If $T \in \bbT(\theta)$, then $f_T$ is supported in $3\theta$;}
\item[\textnormal{(ii)}]{If $(t,x) \in B_R \setminus T$, then $|\fE_rf_T(t,x)| \leq R^{-1000}\|f\|_2$;}
\item[\textnormal{(iii)}]{$|\fE_rf(t,x) - \sum_{T \in \bbT}\fE_rf_T(t,x)| \leq R^{-1000}\|f\|_2$ for every $(t,x) \in B_R$;}
\item[\textnormal{(iv)}]{If $T_1,T_2 \in \bbT(\theta)$ and $T_1 \cap T_2 = \emptyset$, then $|\int f_{T_1}\overline{f_{T_2}}| \leq R^{-1000}\|f\|_{L^2(\theta)}^2$;}
\item[\textnormal{(v)}]{$\sum_{T \in \bbT(\theta)}\|f_T\|_2^2 \lesssim \|f\|_{L^2(\theta)}^2$.}
\end{itemize}
\end{proposition}
\begin{proof}
Adapting Guth's argument in \cite{Guth} is straightforward.  The fact that the derivatives of $\phi_r$ are bounded in $r$ (i.e.~$\sup_{\xi \in U}|\nabla^k\phi_r(\xi)| \lesssim_k 1$) ensures that all constants arising in the argument can be made uniform in $r$.  We note, in particular, that the crucial derivative estimates appearing in line (17) of \cite{Guth} hold uniformly in $r$ when adapted to our setting.
\end{proof}

Next, we record an orthogonality lemma from \cite{Guth}.  The special case $N =1$ will be of particular use.  
\begin{lemma}\label{lem3.2}
Let $\bbT_1,\ldots,\bbT_N$ be subsets of $\bbT$. Suppose that each tube in $\bbT$ belongs to at most $M$ of the $\bbT_i$, and for each $\tau \in \fS_r$, let
\begin{align*}
f_{\tau,i} := \sum_{T \in \bbT_i}f_{\tau,T},
\end{align*}
where the functions $f_{\tau,T}$ come from applying Proposition \ref{prop3.1} to $f_\tau$.  Then
\begin{align*}
\sum_{i=1}^N\int_{3\theta}|f_{\tau,i}|^2 \lesssim M\int_{10\theta}|f_\tau|^2
\end{align*}
for every $\theta \in \Theta$, and
\begin{align*}
\sum_{i =1}^N\int_U|f_{\tau,i}|^2 \lesssim M\int_U|f_\tau|^2.
\end{align*}
\end{lemma}

Finally, we turn to polynomial partitioning.  Let $P$ be a polynomial on $\R^d$.  We denote the zero set of $P$ by $Z(P)$ and say that $z \in Z(P)$ is nonsingular if $\nabla P(z) \neq 0$.  If $z$ is nonsingular, then $Z(P)$ is a smooth hypersurface near $z$.  If every point of $Z(P)$ is nonsingular, then we say that $P$ is nonsingular.
 
\begin{proposition}[Guth \cite{Guth}]\label{prop3.2}
Given $g \in L^1(\R^d)$ and $D \geq 1$, there exists a polynomial $P$ of degree at most $D$ such that $P$ is a product of nonsingular polynomials and each connected component $O$ of $\R^d \setminus Z(P)$ satisfies
\begin{align*}
\int_{O} |g| \sim \frac{1}{D^d}\int_{\R^d} |g|.
\end{align*}
\end{proposition}
\noindent We note that a product of nonsingular polynomials may have singular points. However, by a perturbation argument using Sard's theorem, one can ensure that nonsingular points are dense in the zero set of the partitioning polynomial.

\subsection{Main proof}
We are now ready to prove Theorem \ref{thm2.4} in earnest.  We will induct on $R$ and $\sum_{\tau \in \fS_r}\|f_\tau\|_2^2$.  The base cases, that $R \sim 1$ or $\sum_{\tau}\|f_\tau\|_2^2 \leq R^{-1000}$, are easy to check, and our induction hypotheses are that Theorem \ref{thm2.4} holds with: (i) $R/2$ in place of $R$, or (ii) $g$ in place of $f$ whenever $\sum_{\tau}\|g_\tau\|_2^2 \leq \frac{1}{2}\sum_{\tau}\|f_\tau\|_2^2$.  Throughout the proof, we will assume that $\varepsilon$ is sufficiently small and that $R$ is sufficiently large in relation to $\varepsilon$. 

We begin by setting $D := R^{\varepsilon^4}$ and applying Proposition \ref{prop3.2} to the function $|\Br_\alpha \fE_r f|^{13/4}\chi_{B_R}$ to produce a polynomial $P$ of degree at most $D$ such that
\begin{align*}
\R^3 \setminus Z(P) = \bigcup_{i \in I}O_i,
\end{align*}
where the ``cells" $O_i$ are connected, pairwise disjoint, and satisfy
\begin{align}\label{ex4.1}
\int_{B_R \cap O_i}|\Br_\alpha\fE_rf|^{13/4} \sim \frac{1}{D^3}\int_{B_R}|\Br_\alpha \fE_r f|^{13/4}.
\end{align}
In particular, the number of cells is $\#I \sim D^3$.  We define the ``wall" $W$ as the $R^{1/2+\delta}$-neighborhood of $Z(P)$, and we set $O_i' := O_i \setminus W$.  Thus,
\begin{align}\label{ex4.2}
\int_{B_R}|\Br_\alpha\fE_rf|^{13/4} = \sum_{i \in I}\int_{B_R \cap O_i'}|\Br_\alpha\fE_rf|^{13/4} + \int_{B_R \cap W}|\Br_\alpha\fE_r f|^{13/4}.
\end{align}
We now argue by cases, according to which term on the right-hand side of \eqref{ex4.2} dominates.

\subsection{Cellular case}
Suppose that the total contribution from the shrunken cells $O_i'$ dominates.  In this case, we have
\begin{align*}
\int_{B_R}|\Br_\alpha\fE_rf|^{13/4} \lesssim \sum_{i \in I}\int_{B_R \cap O_i'}|\Br_\alpha\fE_rf|^{13/4}.
\end{align*}
Using \eqref{ex4.1}, we then see that the contribution from any single $O_i'$ is controlled by the average of all such contributions.  Thus, ``most" cells should contribute close to the average, and it is straightforward to show that there exists $J \subseteq I$ such that $\#J \sim D^3$ and
\begin{align}\label{ex4.3*}
\int_{B_R \cap O_i'}|\Br_\alpha \fE_r f|^{13/4} \sim \frac{1}{D^3}\int_{B_R}|\Br_\alpha\fE_r f|^{13/4}
\end{align}
for all $i \in J$.  The lower bound on $\#J$ will be the basis for a pigeonholing argument shortly.

First, some definitions are needed.  For each $i \in I$ and $\tau \in \fS_r$, we set
\begin{align*}
\bbT_i := \{T \in \bbT : T \cap O_i' \neq \emptyset\}
\end{align*}
and
\begin{align*}
f_{\tau,i} := \sum_{T \in \bbT_i} f_{\tau,T},
\end{align*}
where the functions $f_{\tau,T}$ come from applying Proposition \ref{prop3.1} to $f_\tau$.  We also set
\begin{align*}
f_{i} := \sum_{\tau \in \fS_r} f_{\tau,i}.
\end{align*}
Since $f_\tau$ is supported in $m \tau$, property (i) in Proposition \ref{prop3.1} implies that $f_{\tau,i}$ is supported in an $O(R^{-1/2})$-neighborhood of $m\tau$.  Let $\overline{f}_i := \chi_U f_i$ and $\overline{f}_{i,\tau} := \chi_U f_{i,\tau}$.  If $R$ is sufficiently large, then $\supp \overline{f}_{\tau,i} \subseteq 2m\tau$.  Consequently, $\overline{f}_i$ has a well defined broad part with respect to these larger squares.  Soon we will apply our induction hypothesis to $\overline{f}_i$ (for some special $i$) with $m$ replaced by $2m$.     

\begin{lemma}\label{lem4.1}
If $(t,x) \in O_i'$ and $\alpha \leq 1/2$, then
\begin{align*}
|\Br_\alpha \fE_r f(t,x)| \leq |\Br_{2\alpha}\fE_r \overline{f}_i(t,x)| + R^{-900}\sum_{\tau \in \fS_r}\|f_\tau\|_2.
\end{align*}
\end{lemma}

\begin{proof}
First, we may assume that
\begin{align}\label{ex4.3}
|\fE_rf(t,x)| \geq R^{-900}\sum_\tau\|f_\tau\|_2;
\end{align}
otherwise, the required inequality is trivial.  Since $(t,x) \in O_i'$, properties (iii) and (ii) in Proposition \ref{prop3.1} imply that
\begin{align*}
\fE_r f_\tau(t,x) = \sum_{T \in \bbT_i}\fE_rf_{\tau,T}(t,x) + O(R^{-990}\|f_\tau\|_2)
\end{align*}
for each $\tau$.  Summing over $\tau$, we get
\begin{align}\label{ex4.5*}
\fE_rf(t,x) = \fE_rf_i(t,x) + O\Big(R^{-990}\sum_\tau\|f_\tau\|_2\Big).
\end{align}
Now it suffices to show that if $(t,x)$ is $\alpha$-broad for $f$, then $(t,x)$ is $2\alpha$-broad for $\overline{f}_i$.  Assume the former and fix $\rho \in \fR_r$.  Using Proposition \ref{prop3.1} again, we have
\begin{align*}
\bigg\vert\sum_{\tau : \tau \subseteq \rho}\fE_r\overline{f}_{\tau,i}(t,x)\bigg\vert &= \bigg\vert\sum_{\tau : \tau \subseteq \rho}\fE_rf_{\tau,i}(t,x)\bigg\vert\\
 &= \bigg\vert\sum_{\tau : \tau \subseteq \rho}\fE_rf_\tau(t,x)\bigg\vert + O\Big(R^{-990}\sum_\tau\|f_\tau\|_2\Big)\\
&\leq \alpha|\fE_r f(t,x)| + O\Big(R^{-990}\sum_\tau\|f_\tau\|_2\Big).
\end{align*}
Using \eqref{ex4.3}, \eqref{ex4.5*}, and the fact that $\alpha \geq K^{-\varepsilon}$, the right-hand side is at most $2\alpha|\fE_rf_i(t,x)| = 2\alpha|\fE_r\overline{f}_i(t,x)|$ for $R$ sufficiently large.
\end{proof}

If $\alpha > 1/2$, then the estimate in Theorem \ref{thm2.4} holds trivially, since the power of $R$ can then be made at least $1000$ by taking $\varepsilon$ sufficiently small.  Thus, we may assume that $\alpha \leq 1/2$.  Applying Lemma \ref{lem4.1} to \eqref{ex4.3*} and recalling that $D = R^{\varepsilon^4}$, we get
\begin{align}\label{ex4.5}
\int_{B_R}|\Br_\alpha\fE_rf|^{13/4} \lesssim D^3\int_{B_R \cap O_i'}|\Br_{2\alpha}\fE_r\overline{f}_i|^{13/4} + O\Big(R^{-1000}\Big(\sum_{\tau \in \fS_r}\|f_\tau\|_2\Big)^{13/4}\Big)
\end{align}
for every $i \in J$.  We will now pick $i_0 \in J$ so that $\sum_{\tau \in \fS_r}\|\overline{f}_{\tau,i_0}\|_2^2$ is small, which will allow us to apply our induction hypothesis to $\overline{f}_{i_0}$.  Because $Z(P)$ is the zero set of a polynomial of degree at most $D$, any line is either contained in $Z(P)$ or intersects $Z(P)$ at most $D$ times.  Thus, each tube in $\bbT$ belongs to at most $D+1$ of the sets $\bbT_i$.  Now, applying Lemma \ref{lem3.2} and the bound $\#J \gtrsim D^3$, we must have
\begin{align*}
\frac{1}{\#J}\sum_{i \in J}\sum_\tau\|f_{\tau,i}\|_2^2 \leq \frac{C}{D^2}\sum_\tau\|f_\tau\|_2^2
\end{align*}
for some constant $C$. Consequently, there exists $i_0 \in J$ such that
\begin{align}\label{ex4.6}
\sum_\tau\|\overline{f}_{\tau,i_0}\|_2^2 \leq \sum_\tau\|f_{\tau,i_0}\|_2^2 \leq \frac{C}{D^2}\sum_\tau\|f_\tau\|_2^2 \leq \frac{1}{2}\sum_\tau\|f_\tau\|_2^2
\end{align}
for $R$ sufficiently large.  We can apply Theorem \ref{thm2.4} to $\overline{f}_{i_0}$ with $2m$ in place of $m$, provided \eqref{ex2.1} holds.  Since \eqref{ex2.1} holds for $f$, Lemma \ref{lem3.2} gives
\begin{align*}
\dashint_{D(\xi,R^{-1/2})}|\overline{f}_{\tau,i_0}|^2 \leq \dashint_{D(\xi,R^{-1/2})}|f_{\tau,i_0}|^2 \lesssim \dashint_{D(\xi,100R^{-1/2})}|f_\tau|^2 \lesssim 1.
\end{align*}
Thus, after multiplying $\overline{f}_{i_0}$ by a constant, we can apply Theorem \ref{thm2.4} to \eqref{ex4.5} with $i = i_0$ to get
\begin{align*}
\int_{B_R}|\Br_\alpha\fE_r f|^{13/4} \leq CD^3C_\varepsilon'R^{\varepsilon + \varepsilon^6\log(8K^\varepsilon\alpha m^2)}\Big(\sum_{\tau \in \fS_r}\|\overline{f}_{\tau,i_0}\|_2^2\Big)^{3/2+\varepsilon} + O\Big(R^{-1000}\Big(\sum_{\tau \in \fS_r}\|f_\tau\|_2\Big)^{13/4}\Big)
\end{align*}
for some $C$.  If the big $O$ term dominates, then the desired estimate follows easily.  Assuming it does not, then by \eqref{ex4.6} and the definition of $D$, we have altogether
\begin{align*}
\int_{B_R}|\Br_\alpha \fE_rf|^{13/4} \leq 2CR^{-2\varepsilon^5+\varepsilon^6\log(8)}C_\varepsilon'R^{\varepsilon+\varepsilon^6\log(K^\varepsilon\alpha m^2)}\Big(\sum_{\tau \in \fS_r}\|f_\tau\|_2^2\Big)^{3/2+\varepsilon},
\end{align*}
and the induction closes if $\varepsilon$ is sufficiently small and $R$ sufficiently large.

\subsection{Algebraic case}
Next, suppose that the contribution from $W$ dominates in \eqref{ex4.2}, so that
\begin{align}\label{ex4.8*}
\int_{B_R}|\Br_\alpha \fE_rf|^{13/4} \lesssim \int_{B_R \cap W}|\Br_\alpha\fE_r f|^{13/4}.
\end{align}
Following Guth \cite{Guth}, we distinguish between tubes that intersect $W$ transversely and those essentially tangent to $W$.  Let $\fB$ be a collection of balls $B$ of radius $R^{1-\delta}$ that cover $B_R$ with bounded overlap.

\begin{definition}
Fix $B \in \fB$.  Let $\bbT_B^\flat$ be the set of tubes $T$ satisfying $T \cap W \cap B \neq \emptyset$ and $\angle(v(T),T_z Z(P)) \leq R^{2\delta - 1/2}$ for every nonsingular point $z \in Z(P) \cap 2B \cap 10T$.  Let $\bbT_B^\sharp$ be the set of tubes $T$ satisfying $T \cap W \cap B \neq \emptyset$ and $T \notin \bbT_B^\flat$.
\end{definition}

Observe that if $T$ intersects $W \cap B$, then $T$ belongs to exactly one of $\bbT_B^\flat$ and $\bbT_B^\sharp$.  (The definition of $\bbT_B^\flat$ would be vacuous if $Z(P) \cap 2B \cap 10T$ contained only singular points; however, as noted above, we can arrange for nonsingular points to be dense in $Z(P)$.)  Thus, on $W \cap B$, each $\fE_rf_\tau$ is well approximated by the sum of the ``tangent" and ``transverse" wave packets, $\{\fE_rf_{\tau,T}\}_{T \in \bbT_B^\flat}$ and $\{\fE_rf_{\tau,B}\}_{T \in \bbT_B^\sharp}$, respectively.  Roughly speaking, our desired bound for $\|\Br_\alpha \fE_rf\|_{L^{13/4}(B \cap W)}$ will soon be reduced to a broad part estimate on the transverse contribution and a bilinear estimate on the tangent contribution.  The following geometric lemma, due to Guth \cite{Guth}, will be critical for establishing those bounds:

\begin{lemma}\label{lem4.3}
\textnormal{(a)} Each $T \in \bbT$ belongs to $\bbT_B^\sharp$ for at most $D^{O(1)}$ balls $B \in \fB$. \textnormal{(b)} For each $B \in \fB$, the number of discs $\theta \in \Theta$ such that $\bbT_B^\flat \cap \bbT(\theta) \neq \emptyset$ is at most $R^{O(\delta) + 1/2}$.  
\end{lemma}

To carry out the bilinear argument, we need to define the separation condition mentioned in Section 2.  Recall how we defined the Lorentz separation $\ldist(\xi,\zeta)$ of $\xi,\zeta \in \Omega$.  We say that two squares $\tau_1,\tau_2 \in \fS_r$ are \emph{separated} if
\begin{align*}
\ldist(r\xi, r\zeta) \geq {C_0rm}{K^{-1}}
\end{align*}
for all $\xi \in 2m\tau_1$ and $\zeta \in 2m\tau_2$, where $C_0 \geq 1$ is a constant to be chosen later.  Part (a) of Lemma \ref{lem4.4*} implies that points having small Lorentz separation must lie near a common line.  The next lemma extends this property to collections of non-separated squares.  \begin{lemma}\label{lem4.4}
Let $\fI \subseteq \fS_r$ be a collection of pairwise non-separated squares.  Then there exist $\sigma_1,\ldots,\sigma_4 \in \fT_{\delta,\delta_0,r} \cup \fT_{\delta_0,\delta,r}$, with $K^{-1} \leq \delta \lesssim mK^{-1}$, such that $\tau \subseteq \bigcup_{i=1}^4\sigma_i$ for every $\tau \in \fI$.
\end{lemma}

\begin{proof}
For $\xi \in U$, let $\overline{\xi} := \Phi^{-1}(r\xi)$ and also set
\begin{align*}
{I} &:= \bigcup_{\tau \in I}\tau,\\
\overline{I} &:= \{\overline{\xi} : \xi \in I\}.
\end{align*}
Fix $\tau_1,\tau_2 \in \fI$. By part (a) of Lemma \ref{lem4.4*} and the definition of (non-)separated squares, there exist $\xi^*\in 2m\tau_1$ and $\zeta^* \in 2m\tau_2$ such that
\begin{align*}
\dist(r\xi^*,\ell_{r\zeta^*}^+ \cup \ell_{r\zeta^*}^-) \lesssim {rm}{K^{-1}}.
\end{align*}
Let $\eta$ be a point in $\ell_{r\zeta^*}^+ \cup \ell_{r\zeta^*}^-$ closest to $r\xi^*$.  By elementary geometry, $\eta$ lies in $2U$.  Thus, from the bound $\|\Phi^{-1}\|_{C^1(2U)} \lesssim 1$ and Lemma \ref{lem2.3}, we have
\begin{align*}
\dist(\overline{\xi^*}, (\{\overline{\zeta_1^*}\} \times \R) \cup (\R \times \{\overline{\zeta_2^*}\})) \lesssim |r\xi^* - \eta| \lesssim {rm}{K^{-1}}.
\end{align*}
Since $\diam \Phi^{-1}(r\cdot 2m\tau) \lesssim rmK^{-1}$ for each $\tau$, it follows that
\begin{align}\label{ex4.10*}
\dist(\overline{\xi}, (\{\overline{\zeta}_1\} \times \R) \cup (\R \times \{\overline{\zeta}_2\})) \leq A
\end{align}
for all $\xi,\zeta \in {I}$ and some $A \lesssim rmK^{-1}$.  Fix $\zeta \in {I}$ and set
\begin{align*}
S := [\overline{\zeta}_1-A,\overline{\zeta}_1+A] \times \R,\\
T := \R \times [\overline{\zeta}_2-A,\overline{\zeta}_2+A],
\end{align*}
so that $\overline{I} \subseteq S \cup T$.  Additionally, define
\begin{align*}
 {3S} := [\overline{\zeta}_1-3A,\overline{\zeta}_1+3A] \times \R,\\
 {3T} := \R \times [\overline{\zeta}_2-3A,\overline{\zeta}_2+3A].
\end{align*}
We consider three exhaustive cases:
\begin{itemize}
\item[(i)]{If $\overline{I} \cap (S \setminus 3T) \neq \emptyset$, then \eqref{ex4.10*} implies that $\overline{I} \cap (T \setminus 3S) = \emptyset$, and consequently $\overline{I} \subseteq 3S$.}
\item[(ii)]{If $\overline{I} \cap (T \setminus 3S) \neq \emptyset$, then \eqref{ex4.10*} implies that $\overline{I} \cap (S \setminus 3T) = \emptyset$, and consequently $\overline{I} \subseteq 3T$.}
\item[(iii)]{Otherwise, $\overline{I} \subseteq 3S \cap 3T$.}
\end{itemize}
Thus, by symmetry, we may assume that $\Phi_r^{-1}({I}) = r^{-1}\overline{I} \subseteq (r^{-1}\cdot 3S) \cap [-\delta_0,\delta_0)^2$.  The interval
\begin{align*}
[{r^{-1}}(\overline{\zeta_1} - 3A), {r^{-1}}(\overline{\zeta_1}+3A)] \cap [-\delta_0,\delta_0)
\end{align*}
is covered by two dyadic intervals $I_1,I_2 \subseteq [-\delta_0,\delta_0)$ of length $\delta \lesssim r^{-1}A \lesssim mK^{-1}$.  Thus, if we set
\begin{align*}
\sigma_1 &:= \Phi_r(I_1 \times [-\delta_0,0)) \cap U,\\
\sigma_2 &:= \Phi_r(I_2 \times [-\delta_0,0)) \cap U,\\
\sigma_3 &:= \Phi_r(I_1 \times [0,\delta_0)) \cap U,\\
\sigma_4 &:= \Phi_r(I_2 \times [0,\delta_0)) \cap U,
\end{align*}
then $I \subseteq \bigcup_{i=1}^4 \sigma_i$ and the proof is complete.
\end{proof}

As mentioned above, estimating $\|\Br_\alpha \fE_rf\|_{L^{13/4}(B \cap W)}$ can be reduced to estimating certain contributions from transverse and tangent wave packets.  The next lemma facilitates this reduction.  First, some notation is needed.  For $\tau \in \fS_r$ and $B \in \fB$, we set
\begin{align*}
f_{\tau,B}^\flat := \sum_{T \in \bbT_B^\flat} f_{\tau,T} \quad\quad \text{and}\quad\quad f_{\tau,B}^\sharp := \sum_{T \in \bbT_B^\sharp} f_{\tau,T}.
\end{align*}
We also let
\begin{align*}
f_B^\flat := \sum_{\tau \in \fS_r} f_{\tau,B}^\flat \quad\quad \text{and} \quad\quad f_B^\sharp := \sum_{\tau \in \fS_r} f_{\tau,B}^\sharp.
\end{align*}
Given $\fI \subseteq \fS_r$, we set
\begin{align*}
f_{\fI,B}^\flat := \sum_{\tau \in \fI}f_{\tau,B}^\flat \quad\quad\text{and}\quad\quad f_{\fI,B}^\sharp := \sum_{\tau \in \fI}f_{\tau,B}^\sharp.
\end{align*}
We note that $f_{\fI,B}^\sharp$ (analogously $f_{\fI,B}^\flat$) has the natural decomposition $f_{\fI,B}^\sharp = \sum_{\tau \in \fS_r} f_{\tau,\fI,B}^\sharp$, where
\begin{align*}
f_{\tau,\fI,B}^\sharp := \begin{cases} f_{\tau,B}^\sharp &\text{if }\tau \in \fI,\\
0 &\text{if } \tau \notin \fI. \end{cases}
\end{align*}
Let $\overline{f}_{\fI,B}^\sharp := \chi_U f_{\fI,B}^\sharp$ and $\overline{f}_{\tau,B}^\sharp := \chi_Uf_{\tau,B}^\sharp$.  Then $\supp \overline{f}_{\tau,B}^\sharp \subseteq 2m\tau$, and thus
$\overline{f}_{\fI,B}^\sharp$ has a well defined broad part.  Finally, we  define
\begin{align*}
\Bil(\fE_r f_B^\flat) := \sum_{\tau_1,\tau_2 \textnormal{ separated}}|\fE_r f_{\tau_1,B}^\flat|^{1/2}|\fE_r f_{\tau_2,B}^\flat|^{1/2}.
\end{align*}

\begin{lemma}\label{lem4.5}
If $(t,x) \in B \cap W$ and $\alpha m$ is sufficiently small, then
\begin{align*}
|\Br_\alpha \fE_rf(t,x)| \leq \sum_{\fI \subseteq \fS_r} |\Br_{10\alpha} \fE_r \overline{f}_{\fI,B}^\sharp(t,x)| + K^{100}\Bil(\fE_r f_B^\flat)(t,x) + R^{-900}\sum_{\tau \in \fS_r}\|f_\tau\|_2.
\end{align*}
\end{lemma}
\begin{proof}
We may assume that $(t,x)$ is $\alpha$-broad for $\fE_rf$ and that
\begin{align}\label{ex4.8}
|\fE_r f(t,x)| \geq R^{-900}\sum_\tau \|f_\tau\|_2.
\end{align}
Let
\begin{align*}
\fI := \{\tau \in \fS_r : |\fE_r f(t,x)| \leq K^{100}|\fE_r f_{\tau,B}^\flat(t,x)|\}.
\end{align*}
If $\fI$ contains a pair of separated squares, then the bound $|\Br_\alpha \fE_r f(t,x)| \leq K^{100}\Bil(\fE_r f_B^\flat)(t,x)$ follows immediately.  Thus, we may assume that $\fI$ contains no pair of separated squares.  By Lemma \ref{lem4.4}, there exist $\sigma_1,\ldots,\sigma_4 \in \fT_{\delta,\delta_0,r} \cup \fT_{\delta_0,\delta,r}$, with $K^{-1} \leq \delta \lesssim mK^{-1}$, such that $\tau \subseteq \bigcup_{i=1}^4\sigma_i$ for every $\tau \in \fI$.  Let
\begin{align*}
\fJ := \Big\{\tau \in \fS_r : \tau \subseteq \bigcup_{i=1}^4 \sigma_i\Big\}.
\end{align*}
Then
\begin{align*}
|\fE_r f(t,x)| \leq \sum_{i=1}^4\bigg\vert\sum_{\tau : \tau \subseteq \sigma_i} \fE_rf_\tau(t,x)\bigg\vert + \bigg\vert\sum_{\tau \in \fJ^c} \fE_rf_\tau(t,x)\bigg\vert.
\end{align*}
Since $\delta \lesssim mK^{-1}$, each $\sigma_i$ is a union of at most $Cm$ elements of $\fR_{\delta_0,r}$ where $C$ is a constant.  Thus, since $(t,x)$ is $\alpha$-broad for $\fE_rf$ and $\alpha m$ is sufficiently small, we have
\begin{align*}
\sum_{i=1}^4\bigg\vert\sum_{\tau : \tau \subseteq \sigma_i} \fE_rf_\tau(t,x)\bigg\vert \leq 4Cm\alpha|\fE_r f(t,x)| \leq \frac{1}{10}|\fE_r f(t,x)|,
\end{align*}
and consequently,
\begin{align*}
\frac{9}{10}|\fE_r f(t,x)| \leq \bigg\vert\sum_{\tau \in \fJ^c} \fE_r f_\tau(t,x)\bigg\vert.
\end{align*}
Since $(t,x) \in B \cap W$, properties (iii) and (ii) in Proposition \ref{prop3.1} imply that
\begin{align*}
\fE_r f_\tau(t,x) = \fE_r f_{\tau,B}^\sharp(t,x) + \fE_r f_{\tau,B}^\flat(t,x) + O(R^{-990}\|f_\tau\|_2)
\end{align*}
for every $\tau \in \fS_r$.  Summing over $\tau \in \fJ^c$, we get
\begin{align*}
\bigg\vert\sum_{\tau \in \fJ^c}\fE_r f_\tau(t,x)\bigg\vert \leq |\fE_r f_{\fJ^c,B}^\sharp(t,x)| + |\fE_r f_{\fJ^c,B}^\flat(t,x)| + O\Big(R^{-990}\sum_\tau\|f_\tau\|_2\Big).
\end{align*}
Since $\fJ^c \subseteq \fI^c$, we have
\begin{align*}
|\fE_r f_{\fJ^c,B}^\flat(t,x)| \leq \#\fJ^c K^{-100}|\fE_r f(t,x)| \leq K^{-97}|\fE_rf(t,x)|.
\end{align*}
Hence,
\begin{align*}
\frac{9}{10}|\fE_r f(t,x)| \leq |\fE_r f_{\fJ^c,B}^\sharp(t,x)| + K^{-97}|\fE_r f(t,x)| + O\Big(R^{990}\sum_\tau\|f_\tau\|_2\Big).
\end{align*}
Using \eqref{ex4.8}, we see that
\begin{align*}
|\fE_r f(t,x)| \leq \frac{5}{4}|\fE_r f_{\fJ^c,B}^\sharp(t,x)| = \frac{5}{4}|\fE_r \overline{f}_{\fJ^c,B}^\sharp(t,x)|,
\end{align*}
provided $\varepsilon$ is sufficiently small and $R$ sufficiently large.  To finish the proof, we will show that $(t,x)$ is $10\alpha$-broad for $\fE_r \overline{f}_{\fJ^c,B}^\sharp$.  It suffices to show that
\begin{align}\label{ex4.9}
\bigg\vert\sum_{\tau \in \fJ^c : \tau \subseteq \rho}\fE_r \overline{f}_{\tau,B}^\sharp(t,x)\bigg\vert \leq 8\alpha|\fE_rf(t,x)|
\end{align}
for every $\rho \in \fR_r$.  Fixing $\rho \in \fR_r$, we have
\begin{align*}
\bigg\vert \sum_{\tau \in \fJ^c : \tau \subseteq \rho} \fE_r \overline{f}_{\tau,B}^\sharp(t,x)\bigg\vert &= \bigg\vert \sum_{\tau \in \fJ^c : \tau \subseteq \rho} \fE_r f_{\tau,B}^\sharp(t,x)\bigg\vert\\ 
&\leq \bigg\vert\sum_{\tau \in \fJ^c : \tau \subseteq \rho} \fE_r f_\tau(t,x)\bigg\vert + \bigg\vert\sum_{\tau \in \fJ^c : \tau \subseteq \rho} \fE_r f_{\tau,B}^\flat(t,x)\bigg\vert + O\Big(R^{-990}\sum_\tau \|f_\tau\|_2\Big).
\end{align*}
As above, $\fJ^c \subseteq \fI^c$ implies that
\begin{align*}
\bigg\vert\sum_{\tau \in \fJ^c : \tau \subseteq \rho} \fE_r f_{\tau,B}^\flat(t,x)\bigg\vert \leq K^{-97}|\fE_rf(t,x)| \leq \alpha|\fE_r f(t,x)|.
\end{align*}
It is straightforward to check that $\sigma_i \cap \rho \in \fR_r$ for each $i = 1,\ldots,4$.  Thus, since $(t,x)$ is $\alpha$-broad for $\fE_rf$, we have
\begin{align*}
\bigg\vert\sum_{\tau \in \fJ^c : \tau \subseteq \rho} \fE_r f_\tau(t,x)\bigg\vert \leq \bigg\vert\sum_{\tau :\tau \subseteq \rho} \fE_r f_\tau(t,x)\bigg\vert + \sum_{i=1}^4\bigg\vert\sum_{\tau : \tau \subseteq \sigma_i \cap \rho} \fE_r f_\tau(t,x)\bigg\vert \leq 5\alpha|\fE_rf(t,x)|.
\end{align*}
Using the preceding three estimates and \eqref{ex4.8}, we arrive at \eqref{ex4.9}.
\end{proof}

If $\alpha m \gtrsim 1$ so that Lemma \ref{lem4.5} does not apply, then the estimate in Theorem \ref{thm2.4} holds trivially, since the power of $R$ can then be made at least $1000$ by taking $\varepsilon$ sufficiently small.  Thus, we may assume that $\alpha m \ll 1$.  We now apply Lemma \ref{lem4.5} to \eqref{ex4.8*} to get
\begin{align}\label{ex4.11}
\notag\int_{B_R} |\Br_\alpha \fE_r f|^{13/4} \lesssim \sum_{B \in \fB}\sum_{\fI \subseteq \fS_r}\int_{B \cap W} |\Br_{10\alpha}\fE_r \overline{f}_{\fI,B}^\sharp|^{13/4} + \sum_{B \in \fB}\int_{B \cap W} \Bil(\fE_rf_B^\flat)^{13/4}\\ + R^{-1000}\Big(\sum_{\tau \in \fS_r}\|f_\tau\|_2\Big)^{13/4};
\end{align}
note that the implicit constant may (and does) depend on $K$, a function of $\varepsilon$.  If the last term dominates in \eqref{ex4.11}, then the estimate in Theorem \ref{thm2.4} holds trivially.

\subsubsection{Transverse subcase} Suppose that the first term dominates in \eqref{ex4.11}, so that
\begin{align}\label{ex4.12}
\int_{B_R} |\Br_\alpha \fE_r f|^{13/4} \lesssim \sum_{B \in \fB}\sum_{\fI \subseteq \fS_r}\int_B |\Br_{10\alpha}\fE_r \overline{f}_{\fI,B}^\sharp|^{13/4}.
\end{align}
Each ball $B \in \fB$ has radius $R^{1-\delta}$, so by induction on $R$, we can apply Theorem \ref{thm2.4} to each summand in \eqref{ex4.12}, whenever \eqref{ex2.1} holds.  Since \eqref{ex2.1} holds for $f$, Lemma \ref{lem3.2} gives
\begin{align*}
\dashint_{D(\xi,R^{-1/2})}|\overline{f}_{\tau,B}^\sharp|^2 \leq \dashint_{D(\xi,R^{-1/2})}|f_{\tau,B}^\sharp|^2 \lesssim \dashint_{D(\xi,100R^{-1/2})}|f_\tau|^2 \lesssim 1.
\end{align*}
Thus, after multiplying by a constant, Theorem \ref{thm2.4} implies that
\begin{align*}
\int_{B_R}|\Br_\alpha \fE_r f|^{13/4} \lesssim \sum_{B \in \fB}\sum_{\fI \subseteq \fS_r}C_\varepsilon' R^{(1-\delta)\varepsilon}R^{\varepsilon^6\log(40 K^\varepsilon \alpha m^2)}\Big(\sum_{\tau \in \fS_r}\|\overline{f}_{\tau,B}^\sharp\|_2^2\Big)^{3/2+\varepsilon}. 
\end{align*}
By Lemma \ref{lem4.3}, each $T \in \bbT$ belongs to at most $D^{O(1)}$ sets $\bbT_B^\sharp$.  Therefore, by Lemma \ref{lem3.2}, we have
\begin{align*}
\sum_{B \in \fB}\Big(\sum_{\tau\in\fS_r}\|\overline{f}_{\tau,B}^\sharp\|_2^2\Big)^{3/2+\varepsilon} \leq \Big(\sum_{\tau \in \fS_r}\sum_{B \in \fB}\|f_{\tau,B}^\sharp\|_2^2\Big)^{3/2+\varepsilon} \lesssim D^{O(1)}\Big(\sum_{\tau \in \fS_r}\|f_\tau\|_2^2\Big)^{3/2+\varepsilon}.
\end{align*}
Since $\delta = \varepsilon^2$, $D = R^{\varepsilon^4}$, and the number of subsets $\fI \subseteq \fS_r$ depends only on $K$, we have altogether
\begin{align*}
\int_{B_R}|\Br_\alpha \fE_r f|^{13/4} \leq CR^{-\varepsilon^3+\varepsilon^6\log(40)+O(\varepsilon^4)}C_\varepsilon' R^{\varepsilon+\varepsilon^6\log(K^\varepsilon \alpha m^2)}\Big(\sum_{\tau \in \fS_r}\|f_\tau\|_2^2\Big)^{3/2+\varepsilon}
\end{align*}
for some $C$ (depending on $\varepsilon$).  The power of the first $R$ is negative for $\varepsilon$ sufficiently small, and then the induction closes for $R$ sufficiently large.

\subsubsection{Tangent subcase} In the remaining case, the second term in \eqref{ex4.11} dominates, whence
\begin{align*}
\int_{B_R} |\Br_\alpha \fE_r f|^{13/4} \lesssim \sum_{B \in \fB}\int_{B \cap W} \Bil(\fE_r f_B^\flat)^{13/4}.
\end{align*}
We will bound the right-hand side directly (i.e.~without induction) using basically standard bilinear restriction techniques and Lemma \ref{lem4.3}.  Since $\#\fB = R^{O(\delta)} \leq R^\varepsilon$, it will suffice to prove the following:

\begin{proposition}\label{prop4.7}
For every $B \in \fB$, we have
\begin{align*}
\int_{B \cap W} \Bil(\fE_r f_B^\flat)^{13/4} \lesssim R^{O(\delta)}\Big(\sum_{\tau \in \fS_{r}}\|f_\tau\|_2^2\Big)^{3/2}.
\end{align*}
\end{proposition}

We will need a preliminary lemma.  Fix $B \in \fB$ and let $\fQ$ be a collection of cubes $Q$ of side length $R^{1/2}$ that cover $B \cap W$ with bounded overlap.  For each $Q \in \fQ$, let
\begin{align*}
\bbT_{B,Q}^\flat := \{T \in \bbT_B^\flat : T \cap Q \neq \emptyset\}.
\end{align*}
Henceforth, we will write ``negligible" in place of any quantity of size $O(R^{-990}\sum_{\tau \in \fS_r}\|f_\tau\|_2)$.  In particular, if $(t,x) \in Q$, then
\begin{align}\label{ex4.15}
\fE_r f_{\tau,B}^\flat(t,x) = \sum_{T \in \bbT_{B,Q}^\flat}\fE_r f_{\tau,T}(t,x) + \negl.
\end{align}
It will suffice to bound $\|\Bil(\fE_rf_B^\flat)\|_{L^{13/4}(Q)}$ for each $Q \in \fQ$.  Informally, the tubes in $\bbT_{B,Q}^\flat$ are tangent to $W$ at $Q$ and are thus coplanar.  Dually, the wave packets $\{\fE_rf_{\tau,T}\}_{T \in \bbT_{B,Q}^\flat}$ have Fourier support near a curve formed by the intersection of $\Sigma_r$ and a plane.  Thus, estimating $\|\Bil(\fE_rf_B^\flat)\|_{L^{13/4}(Q)}$ is essentially a two-dimensional bilinear restriction problem, making the $L^4$ argument a natural approach (as done in \cite{Guth}, of course).
\begin{lemma}\label{lem4.8}
For all $Q \in \fQ$ and separated $\tau_1, \tau_2 \in \fS_r$, we have
\begin{align*}
\int_Q |\fE_r f_{\tau_1,B}^\flat|^2|\fE_r f_{\tau_2,B}^\flat|^2 \lesssim R^{O(\delta)-1/2}\bigg(\sum_{T_1 \in \bbT_{B,Q}^\flat}\|f_{\tau_1,T_1}\|_2^2\bigg)\bigg(\sum_{T_2 \in \bbT_{B,Q}^\flat}\|f_{\tau_2,T_2}\|_2^2\bigg) + \negl.
\end{align*}
\end{lemma}
\begin{proof}
Let $\psi_Q$ be a smooth function satisfying $\chi_Q \leq \psi_Q \leq \chi_{2Q}$ and
\begin{align*}
|\hat{\psi}_Q(\tau,\xi)| \lesssim R^{3/2}(1+|(\tau,\xi)|R^{1/2})^{-10^6/\delta}.
\end{align*}
By \eqref{ex4.15} and Plancherel's theorem, we have
\begin{align}\label{ex4.16}
\notag\int_Q|\fE_r f_{\tau_1,B}^\flat|^2|\fE_rf_{\tau_2,B}^\flat|^2 &= \sum_{T_1,\overline{T}_1,T_2,\overline{T}_2 \in \bbT_{B,Q}^\flat} \int_Q \fE_rf_{\tau_1,T_1}\overline{\fE_rf_{\tau_1,\overline{T}_1}}\fE_r f_{\tau_2,T_2}\overline{\fE_rf_{\tau_2,\overline{T}_2}} + \negl\\
&\leq \sum_{T_1,\overline{T}_1,T_2,\overline{T}_2 \in \bbT_{B,Q}^\flat} \int_{\R^3}(\hat{\psi}_Q \ast d\sigma_{\tau_1,T_1} \ast d\sigma_{\tau_2,T_2})(\overline{d\sigma_{\tau_1,\overline{T}_1}\ast d\sigma_{\tau_2,\overline{T}_2}}) + \negl,
\end{align}
where $d\sigma_{\tau,T}$ is the measure on $\Sigma_r$ defined by
\begin{align}\label{measure}
\int_{\Sigma_r} g d\sigma_{\tau,T} := \int_U g(\phi_r(\xi),\xi)f_{\tau,T}(\xi)d\xi.
\end{align}

Fix $T_1,\overline{T}_1,T_2,\overline{T}_2 \in \bbT_{B,Q}^\flat$ and let $\xi,\overline{\xi},\zeta,\overline{\zeta}$ denote the centers of $\theta(T_1),\theta(\overline{T}_1),\theta(T_2),\theta(\overline{T}_2)$, respectively.  The rapid decay of $\hat{\psi}_Q$ and the fact that $\supp \chi_Uf_{\tau,T} \subseteq \frac{3}{2}m\tau$ for every $\tau \in \fS_r$ and $T \in \bbT$ imply that the contribution of $T_1,\overline{T}_1,T_2,\overline{T}_2$ to \eqref{ex4.16} is negligible unless
\begin{align*}
\xi+\zeta &= \overline{\xi} + \overline{\zeta} + O(R^{\delta-1/2}),\\
\phi_r(\xi) + \phi_r(\zeta) &= \phi_r(\overline{\xi}) + \phi_r(\overline{\zeta}) + O(R^{\delta-1/2}),
\end{align*}
and $\xi,\overline{\xi} \in 2m\tau_1$ and $\zeta,\overline{\zeta} \in 2m\tau_2$.  We need to estimate the number of non-negligible terms in \eqref{ex4.16} involving given tubes $T_1,T_2$.

Toward that end, we adapt some techniques of Cho--Lee \cite{Cho--Lee} and Lee \cite{Lee}.  Assuming $T_1,\overline{T}_1,T_2,\overline{T}_2$ contribute non-negligibly, then
\begin{align}\label{ex4.17}
\phi_r(\xi) + \phi_r(\zeta) = \phi_r(\overline{\xi}) + \phi_r(\xi+\zeta-\overline{\xi}) + O(R^{\delta-1/2}).
\end{align}
We define a function $\Psi : U \rightarrow \R$ by
\begin{align*}
\Psi(\eta) := \phi_r(\eta) + \phi_r(\xi+\zeta-\eta) - \phi_r(\xi)-\phi_r(\zeta)
\end{align*}
and denote by $Z := \Psi^{-1}(0)$ its zero set.  We claim that $|\nabla\Psi| \gtrsim 1$ on $2m\tau_1$.  Indeed, if $\eta \in 2m\tau_1$, then by the Cauchy--Schwarz inequality, boundedness of $\|(\nabla^2 \phi)^{-1}\|$ on $U$, part (b) of Lemma \ref{lem4.4*}, and finally the separation of $\tau_1$ and $\tau_2$, we have
\begin{align*}
|\nabla\phi_r(\eta) - \nabla\phi_r(\zeta)| &= r^{-1}|\nabla\phi(r\eta)-\nabla\phi(r\zeta)|\\
&\gtrsim r^{-1}|\langle(\nabla^2\phi(r\eta))^{-1}(\nabla\phi(r\eta)-\nabla\phi(r\zeta)),\nabla\phi(r\eta)-\nabla\phi(r\zeta)\rangle|^{1/2}\\
&\sim r^{-1}\ldist(r\eta,r\zeta)\\
&\geq C_0mK^{-1},
\end{align*}
whence
\begin{align*}
|\Psi(\eta)| = |\nabla\phi_r(\eta) - \nabla\phi_r(\xi+\zeta-\eta)| \geq |\nabla\phi_r(\zeta)-\nabla\phi_r(\zeta)| - \|\phi\|_{C^1(U)}\diam(2m\tau_1) \gtrsim 1
\end{align*}
if $C_0$ is sufficiently large.  By the claim, $Z$ is a smooth curve near $\xi$, and \eqref{ex4.17} and a Taylor approximation argument imply that
\begin{align}\label{ex4.18}
\dist(\overline{\xi},Z) \lesssim R^{\delta-1/2}
\end{align}
for $R$ sufficiently large.  As mentioned above, tubes in $\bbT_{B,Q}^\flat$ are nearly coplanar.  Inspecting the definition, it is straightforward to check that $\angle(v(T),T_z Z(P)) \leq R^{2\delta-1/2}$ for all $T \in \bbT_{B,Q}^\flat$ and some (nonsingular) $z \in 2R^\delta Q \cap Z(P)$.  Thus, dually, there exists a plane $\Pi$ through the origin such that $\dist((-1,\nabla\phi_r(\eta)),\Pi) \lesssim R^{2\delta-1/2}$ for each $\eta \in \{\xi,\overline{\xi},\zeta\}$.  Consequently, there exists a line whose $O(R^{2\delta-1/2})$-neighborhood contains $\nabla\phi_r(\xi)$, $\nabla\phi_r(\overline{\xi})$, and $\nabla\phi_r(\zeta)$.  Since $|\nabla\phi_r(\xi)-\nabla\phi_r(\zeta)| \gtrsim 1$ due to the separation of $\tau_1$ and $\tau_2$, it follows that $\nabla\phi_r(\overline{\xi})$ lies in an $O(R^{2\delta-1/2})$-neighborhood of the line $\ell$ containing $\nabla\phi_r(\xi)$ and $\nabla\phi_r(\zeta)$.  We consider now the smooth curve $\tilde{\ell} := (\nabla\phi_r)^{-1}(\ell \cap 3U)$, noting that $\nabla\phi$ (and thus $\nabla\phi_r$) is invertible near the origin since $\det\nabla^2\phi(0) \neq 0$.  It contains $\xi$ by construction, and the boundedness of $\|(\nabla^2\phi)^{-1}\|$ implies that
\begin{align}\label{ex4.19}
\dist(\overline{\xi},\tilde{\ell}) \lesssim R^{2\delta-1/2}.
\end{align}
Crucially, $\tilde{\ell}$ and $Z$ intersect transversely at $\xi$.  Indeed, parametrizing $\tilde{\ell}$ by 
\begin{align*}
\tilde{\ell}(t) := (\nabla\phi_r)^{-1}((1-t)\nabla\phi_r(\xi) + t\nabla\phi_r(\zeta)),
\end{align*}
the tangent line to $\tilde{\ell}$ at $\xi$ is parallel to
\begin{align*}
\frac{d}{dt}\tilde{\ell}(t)\bigg\vert_{t=0} = (\nabla^2\phi_r(\xi))^{-1}(\nabla\phi_r(\zeta)-\nabla\phi_r(\xi)),
\end{align*}
and the normal line to $Z$ at $\xi$ is parallel to $\nabla\Psi(\xi) = \nabla\phi_r(\xi) - \nabla\phi_r(\zeta)$.  Thus, the bound
\begin{align*}
|\langle(\nabla^2\phi_r(\xi))^{-1}(\nabla\phi_r(\xi)-\nabla\phi_r(\zeta)),\nabla\phi_r(\xi)-\nabla\phi_r(\zeta)\rangle| \gtrsim 1,
\end{align*}
which follows from part (b) of Lemma \ref{lem4.4*} and the separation of $\tau_1$ and $\tau_2$, implies the claimed transverse intersection.  Consequently, by \eqref{ex4.18} and \eqref{ex4.19}, we have $|\xi - \overline{\xi}| \lesssim R^{2\delta-1/2}$.  A similar argument shows that $|\zeta - \overline{\zeta}| \lesssim R^{2\delta-1/2}$.  Since $\#(\bbT_{B,Q}^\flat \cap \bbT(\theta)) \lesssim 1$ for every $\theta \in \Theta$, it follows that for each $T_1,T_2 \in \bbT_{B,Q}^\flat$, there are $O(R^{8\delta})$ pairs $\overline{T}_1,\overline{T}_2 \in \bbT_{B,Q}^\flat$ such that $T_1,\overline{T_1},T_2,\overline{T}_2$ contribute non-negligibly to \eqref{ex4.16}.

Hence, by the Cauchy--Schwarz inequality (a few times) and Young's inequality, \eqref{ex4.16} is at most
\begin{align}\label{ex4.20}
R^{O(\delta)}\sum_{T_1,T_2 \in \bbT_{B,Q}^\flat} \int_{\R^3}|d\sigma_{\tau_1,T_1}\ast d\sigma_{\tau_2,T_2}|^2 + \negl.
\end{align}
To estimate the convolution, we use Plancherel's theorem and the familiar wave packet approximation
\begin{align}\label{ex4.21}
|\fE_r g_T| \approx R^{-1/2}\|g_T\|_2\chi_T;
\end{align}
we will give a rigorous argument in Lemma \ref{lem4.12}, appearing at the end of the article.  If $T_1,T_2 \in \bbT$ are such that $3\theta(T_i) \cap \tau_i \neq \emptyset$, then the separation of $\tau_1$ and $\tau_2$ implies that the directions $v(T_1)$ and $v(T_2)$ are transverse and consequently that $|T_1 \cap T_2| \lesssim R^{3\delta + 3/2}$.  Hence, by Plancherel's theorem and \eqref{ex4.21}, we (essentially) have 
\begin{align*}
\int_{\R^3}|d\sigma_{\tau_1,T_1} \ast d\sigma_{\tau_2,T_2}|^2 = \int_{\R^3}|\fE_r f_{\tau_1,T_1}\fE_rf_{\tau_2,T_2}|^2 \lesssim R^{3\delta-1/2}\|f_{\tau_1,T_1}\|_2^2\|f_{\tau_2,T_2}\|_2^2.
\end{align*}
Plugging this estimate into \eqref{ex4.20}, we obtain the lemma.
\end{proof}

Given Lemma \ref{lem4.8}, the rest of the proof of Proposition \ref{prop4.7} very closely resembles the corresponding argument in \cite{Guth}.  For the convenience of the reader, we repeat the details here.  We set
\begin{align*}
S_{\tau,B}^\flat := \bigg(\sum_{T \in \bbT_B^\flat} (R^{-1/2}\|f_{\tau,T}\|_2\chi_{2T})^2\bigg)^{1/2}
\end{align*}
(cf.~\eqref{ex4.21}).  Let $\tau_1, \tau_2 \in \fS_r$ be separated squares.  Lemma \ref{lem4.8} implies that
\begin{align*}
\int_Q |\fE_r f_{\tau_1,B}^\flat|^2|\fE_r f_{\tau_2,B}^\flat|^2 \lesssim R^{O(\delta)}\int_Q (S_{\tau_1,B}^\flat)^2(S_{\tau_2,B}^\flat)^2 + \negl.
\end{align*}
Summing over $Q \in \fQ$ and exploiting the separation of $\tau_1$ and $\tau_2$ (as above) leads to the bound
\begin{align*}
\int_{B \cap W} |\fE_r f_{\tau_1,B}^\flat|^2|\fE_r f_{\tau_2,B}^\flat|^2 \lesssim R^{O(\delta)-1/2}\bigg(\sum_{T_1 \in \bbT_B^\flat}\|f_{\tau_1,T_1}\|_2^2\bigg)\bigg(\sum_{T_2 \in \bbT_B^\flat}\|f_{\tau_2,T_2}\|_2^2\bigg) + \negl.
\end{align*}
By properties (i) and (iv) of Proposition \ref{prop3.1}, the functions $f_{\tau,T}$ are nearly orthogonal  and we have
\begin{align*}
\sum_{T \in \bbT_B^\flat}\|f_{\tau,T}\|_2^2 \lesssim \|f_{\tau,B}^\flat\|_2^2+\negl
\end{align*}
for every $\tau$.  Thus, altogether, 
\begin{align*}
\int_{B \cap W} |\fE_r f_{\tau_1,B}^\flat|^2|\fE_r f_{\tau_2,B}^\flat|^2 \lesssim R^{O(\delta)-1/2}\|f_{\tau_1,B}^\flat\|_2^2\|f_{\tau_2,B}^\flat\|_2^2 + \negl,
\end{align*}
and consequently by H\"older's inequality,
\begin{align*}
\|\Bil(\fE_r f_B^\flat)\|_{L^4(B \cap W)} \lesssim R^{O(\delta)-1/8}\bigg(\sum_{\tau \in \fS_r}\|f_{\tau,B}^\flat\|_2^2\bigg)^{1/2} + \negl.
\end{align*}
The well-known estimate
\begin{align*}
\|\fE_rg\|_{L^2(B_R)} \lesssim R^{1/2}\|g\|_2
\end{align*}
(which is a consequence of Plancherel's theorem for the spatial Fourier transform), together with H\"older's inequality, implies that
\begin{align*}
\|\Bil(\fE_rf_B^\flat)\|_{L^2(B \cap W)} \lesssim R^{1/2}\bigg(\sum_{\tau \in \fS_r}\|f_{\tau,B}^\flat\|_2^2\bigg)^{1/2}.
\end{align*}
Hence, by interpolation,
\begin{align}\label{ex4.22}
\int_{B \cap W}\Bil(\fE_r f_B^\flat)^p \lesssim R^{O(\delta) + \frac{5}{2}-\frac{3p}{4}}\bigg(\sum_{\tau \in \fS_r}\|f_{\tau,B}^\flat\|_2^2\bigg)^{p/2} + \negl
\end{align}
for $p \in [2,4]$.  Now, on one hand, $\|f_{\tau,B}^\flat\|_2 \lesssim \|f_\tau\|_2$ by Lemma \ref{lem3.2}.  On the other hand, Lemma \ref{lem4.3} gives a different bound:  There are at most $R^{O(\delta)+1/2}$ discs $\theta \in \Theta$ such that $\bbT_B^\flat \cap \bbT(\theta) \neq \emptyset$.  By property (i) of Proposition \ref{prop3.1}, each $f_{\tau,B}^\flat$ is therefore supported in $R^{O(\delta)+1/2}$ discs $\theta$, on each of which we have the bound
\begin{align*}
\int_\theta |f_{\tau,B}^\flat|^2 \lesssim \int_{10\theta}|f_\tau|^2 \lesssim R^{-1},
\end{align*}
by Lemma \ref{lem3.2} and \eqref{ex2.1}.  Thus, $\|f_{\tau,B}^\flat\|_2 \lesssim R^{O(\delta)-1/4}$.  Combining these two estimates gives $\|f_{\tau,B}^\flat\|_2 \lesssim \|f_\tau\|_2^{{3/p}}R^{O(\delta)-\frac{1}{4}(1-\frac{3}{p})}$ for $p \geq 3$.  Plugging this bound into \eqref{ex4.22} yields
\begin{align*}
\int_{B \cap W}\Bil(\fE_rf_B^\flat)^p \lesssim R^{O(\delta) + \frac{13}{4}-p}\bigg(\sum_{\tau \in \fS_r}\|f_\tau\|_2^2\bigg)^{3/2},
\end{align*}
and then taking $p = 13/4$ completes the proof of Proposition \ref{prop4.7}.

To conclude the article, we rigorously prove the convolution estimate used in the proof of Lemma \ref{lem4.8}.  This standard argument is sketched in \cite{Guth}; we fill in the details here.
\begin{lemma}\label{lem4.12}
If $\tau_1,\tau_2 \in \fS_r$ are separated squares and $T_1,T_2 \in \bbT$ are such that $3\theta(T_i) \cap \tau_i \neq \emptyset$, then
\begin{align*}
\int_{\R^3}|d\sigma_{\tau_1,T_1} \ast d\sigma_{\tau_2,T_2}|^2 \lesssim R^{-1/2}\|f_{\tau_1,T_1}\|_2\|f_{\tau_2,T_2}\|_2,
\end{align*}
where $d\sigma_{\tau_i,T_i}$ is given by \eqref{measure}.
\end{lemma}

\begin{proof}
Let $\theta_i := \theta(T_i)$ and $c_i := c_{\theta_i}$.  Since $3\theta_i \cap \tau_i \neq \emptyset$, we have $c_i \in 2m\tau_i$, and consequently, $|\nabla\phi_r(c_1) - \nabla\phi_r(c_2)| \gtrsim 1$ by the separation of $\tau_1$ and $\tau_2$.  Indeed, by the Cauchy--Schwarz inequality, boundedness of $\|(\nabla^2\phi)^{-1}\|$ on $U$, and part (b) of Lemma \ref{lem4.4*},
\begin{align*}
|\nabla\phi_r(c_1) - \nabla\phi_r(c_2)| &= r^{-1}|\nabla\phi(rc_1)-\nabla\phi(rc_2)|\\
 &\gtrsim r^{-1}|\langle(\nabla^2\phi(rc_1))^{-1}(\nabla\phi(rc_1)-\nabla\phi(rc_2)),\nabla\phi(rc_1)-\nabla\phi(rc_2)\rangle|^{1/2}\\
 &\sim r^{-1}\ldist(rc_1,rc_2)\\
 &\gtrsim 1.
\end{align*}
It follows (from the law of sines, say) that the unit normal vectors $n_1 := v_{\theta_1}$ and $n_2 := v_{\theta_2}$ satisfy $\angle(n_1,n_2) \gtrsim 1$.  Using this angle bound, we will foliate $3\theta_1$ by lines whose lifts to $\Sigma_r$ are transverse to the tangent plane $T_{(\phi_r(c_2),c_2)}\Sigma_r$ above $c_2$.  Define the direction set
\begin{align*}
V := \{\omega \in \mathbb{S}^2 : \omega \cdot n_1 = 0~\text{and}~|\omega \cdot n_2| \geq c\},
\end{align*}
where $c > 0$.  If $c$ is sufficiently small relative to $\angle(n_1,n_2)$, then $V$ is nonempty.  Choose $\omega \in V$, let $\overline{\omega} := (\omega_2,\omega_3)$, and let $S$ be the rotation of $\R^2$ satisfying $S(0,1) = \overline{\omega}/|\overline{\omega}|$ (note that $\overline{\omega} \neq 0$).  Define the lines $\overline{\gamma}_s$ by
\begin{align*}
\overline{\gamma}_s(t) := S(s,t) + c_1,
\end{align*}
and note that $\supp d\sigma_{\tau_1,T_1} \subseteq 3\theta_1 \subseteq \{\overline{\gamma}_s(t) : (s,t) \in I^2\}$, where $I := [-3R^{-1/2},3R^{-1/2}]$.  The lift of $\overline{\gamma}_s$ to $\Sigma_r$ is given by
\begin{align*}
\gamma_s(t) := (\phi_r(\overline{\gamma}_s(t)), \overline{\gamma}_s(t))
\end{align*}
for $s,t$ small.  For almost every $s$, the function $t \mapsto f_{\tau_1,T_1}(\overline{\gamma}_s(t))$ is measurable and
\begin{align*}
\int_{\gamma_s}g d\nu_s := \int_I g(\gamma_s(t))f_{\tau_1,T_1}(\overline{\gamma}_s(t))dt
\end{align*}
defines a measure $d\nu_s$ on $\gamma_s$.  Using \eqref{measure}, an easy calculation shows that $d\sigma_{\tau_1,T_1} = d\nu_s \chi_Ids$.

Now, to prove the required convolution estimate, it suffices to show that
\begin{align*}
|\langle d\sigma_{\tau_1,T_1} \ast d\sigma_{\tau_2,T_2}, \psi \rangle| \lesssim R^{-1/4}\|f_{\tau_1,T_1}\|_2\|f_{\tau_2,T_2}\|_2\|\psi\|_2
\end{align*}
for all $\psi \in C_c^\infty(\R^3)$; the brackets denote the pairing between distributions and test functions.  We compute that
\begin{align*}
|\langle d_{\tau_1,T_1} \ast d\sigma_{\tau_2,T_2}, \psi\rangle| &= \bigg\vert\int_{\Sigma_r}\int_{\Sigma_r}\psi(\sigma+\tau, \zeta+\xi) d\sigma_{\tau_2,T_2}(\sigma,\zeta) d\sigma_{\tau_1,T_1}(\tau,\xi)\bigg\vert\\
&= \bigg\vert \int_I\int_{\gamma_s}\int_{\Sigma_r}\psi(\sigma+\tau,\zeta+\xi)d\sigma_{\tau_2,T_2}(\sigma,\zeta)d\nu_s(\tau,\xi) ds\bigg\vert\\
&\lesssim R^{-1/4}\bigg(\int_I\bigg\vert\int_{\gamma_s}\int_{\Sigma_r}\psi(\sigma+\tau,\zeta+\xi)d\sigma_{\tau_2,T_2}(\sigma,\zeta)d\nu_s(\tau,\xi)\bigg\vert^2 ds\bigg)^{1/2}.
\end{align*}
Using the definitions of $d\sigma_{\tau_2,T_2}$ and $d\nu_s$ and the Cauchy--Schwarz inequality, the quantity between absolute value signs is at most
\begin{align*}
\|f_{\tau_2,T_2}\|_2\bigg(\int_I\int_{3\theta_2}|\psi((\phi_r(\zeta),\zeta)+\gamma_s(t))|^2d\zeta dt\bigg)^{1/2}\bigg(\int_I |f_{\tau_1,T_1}(\overline{\gamma}_s(t))|^2dt\bigg)^{1/2}.
\end{align*}
Thus, if we can show that
\begin{align*}
\int_I\int_{3\theta_2}|\psi((\phi_r(\zeta),\zeta)+\gamma_s(t))|^2d\zeta dt \lesssim \|\psi\|_2^2,
\end{align*}
then a simple change of variable, using the definition of $\overline{\gamma}_s$, gives the required  estimate.

Toward that end, let $G(\zeta,t) := (\phi_r(\zeta),\zeta) +  \gamma_s(t)$.  We claim that $G$ is invertible on $3\theta_2 \times I$, provided $R$ is sufficiently large.  The definition of $S$ implies that $\overline{\gamma}_s'(t) = \overline{\omega}/|\overline{\omega}|$ for every $s,t$.  Thus, the Jacobian of $G$ at $(c_2,0)$ is given by
\begin{align*}
\nabla G(c_2,0) = \left(\begin{array}{ccc}
\partial_1\phi_r(c_2) & \partial_2\phi_r(c_2) & \nabla\phi_r(\gamma_s(0))\cdot \overline{\omega}/|\overline{\omega}|\\ 1 & 0 &\omega_2/|\overline{\omega}| \\ 0 & 1 & \omega_3/|\overline{\omega}|\end{array}\right).
\end{align*}
The first two columns of this matrix are orthogonal to $n_2$.  If we replace $\gamma_s(0)$ by $c_1$, then the third column becomes $\omega/|\overline{\omega}|$, since $\omega \cdot n_1 = 0$.  The angle between $\omega$ and the orthogonal complement of $n_2$ is bounded below, since $|\omega\cdot n_2| \geq c$.   Combining these observations, we see that
\begin{align*}
|\det \nabla G(c_2,0)| = \frac{1}{|\overline{\omega}|}\left\vert\det\left(\begin{array}{ccc}
\partial_1\phi_r(c_2) & \partial_2\phi_r(c_2) & \omega_1\\ 1 & 0 &\omega_2 \\ 0 & 1 & \omega_3\end{array}\right)\right\vert + O(R^{-1/2}) \gtrsim 1.
\end{align*}
Thus, the inverse function theorem implies that $G$ is invertible on $3\theta_2 \times I$, if $R$ is sufficiently large. (The meaning of ``sufficiently large" does not depend on $r$ or $s$, since the bounds $\|\nabla G(c_2,0)\| \sim 1$ and $\|(\nabla G(c_2,0))^{-1}\| \sim 1$ hold uniformly in these parameters.)  Additionally, the bound $|\det\nabla G(\zeta,t)| \gtrsim 1$ holds on $3\theta_2 \times I$, so we obtain
\begin{align*}
\int_I\int_{3\theta_2}|\psi((\phi_r(\zeta),\zeta)+\gamma_s(t))|^2d\zeta dt = \iint_{G(3\theta_2\times I)}|\psi(\eta)|^2|\det\nabla G^{-1}(\eta)|d\eta \lesssim \|\psi\|_2^2,
\end{align*}
completing the proof.
\end{proof}

\end{document}